\documentclass[a4paper,12pt,oneside]{amsart}

\usepackage{amssymb}  
\usepackage{latexsym} 
\usepackage{comment}
\usepackage{color}
\usepackage{colonequals}

 \usepackage[usenames,dvipsnames]{pstricks}
 \usepackage{epsfig}
 \usepackage{pst-grad} 
 \usepackage{pst-plot} 

\definecolor{darkgreen}{rgb}{0,0.5,0}
\usepackage[
        colorlinks, citecolor=darkgreen,
        pdfauthor={Ingrid Bauer, Fabrizio Catanese},
        pdftitle={Rigidity of the Kummer covering},
        linktocpage        
]{hyperref}

\usepackage[alphabetic,backrefs,lite]{amsrefs} 

\usepackage[all]{xy} 

\usepackage{enumerate} 

\newcommand\sC{{\mathcal C}}

\newcommand\sL{{\mathcal L}}

\newcommand\sB{{\mathcal B}}

\newcommand\la{\lambda}
\newcommand\Lam{\Lambda}

\newcommand\e{\epsilon}
\newcommand\s{\sigma}

\newcommand\Ga{\Gamma}

\newcommand\de{\delta}

\DeclareMathOperator{\Pic}{Pic}

\DeclareMathOperator{\Ext}{Ext}
\DeclareMathOperator{\Hom}{Hom}

\DeclareMathOperator{\Def}{Def}
\DeclareMathOperator{\Supp}{Supp}

\DeclareMathOperator{\kod}{kod}

\newcommand{\CC}{\ensuremath{\mathbb{C}}}
\newcommand{\RR}{\ensuremath{\mathbb{R}}}
\newcommand{\ZZ}{\ensuremath{\mathbb{Z}}}
\newcommand{\QQ}{\ensuremath{\mathbb{Q}}}

\newcommand{\sS}{\ensuremath{\mathcal{S}}}

\newcommand{\NN}{\ensuremath{\mathbb{N}}}
\newcommand{\hol}{\ensuremath{\mathcal{O}}}

\newcommand{\HH}{\ensuremath{\mathbb{H}}}
\newcommand{\BB}{\ensuremath{\mathbb{B}}}
\newcommand{\PP}{\ensuremath{\mathbb{P}}}

\newcommand{\FF}{\ensuremath{\mathbb{F}}}

\newcommand{\ra}{\ensuremath{\rightarrow}}

\def\eea{\end{eqnarray*}}
\def\bea{\begin{eqnarray*}}

\newcommand{\Proof}{{\it Proof. }}

\newcommand\dual{\mathrel{\raise3pt\hbox{$\underline{\mathrm{\thinspace d
\thinspace}}$}}}
\newcommand\qe{\ifhmode\unskip\nobreak\fi\quad $\Box$}       

\def\BOX{\hfill\lower.5\baselineskip\hbox{$\Box$}}

\newtheorem{theorem}{Theorem}[section]
\newtheorem{lemma}[theorem]{Lemma}

\newtheorem{proposition}[theorem]{Proposition}
\newtheorem{prop}[theorem]{Proposition}

\newtheorem{question}[theorem]{Question}

\theoremstyle{remark}
\newtheorem{remark}[theorem]{Remark}

\theoremstyle{definition}
\newtheorem{definition}[theorem]{Definition}

\DeclareMathOperator{\PSL}{PSL}

\DeclareMathOperator{\Aut}{Aut}

\DeclareMathOperator{\rk}{rk}

\DeclareMathOperator{\ord}{ord}

\numberwithin{equation}{section}

\newcounter{nootje}
\renewcommand\check[1]
  {\marginpar{\tiny\begin{minipage}{20mm}\begin{flushleft}\thenootje : #1\end{flushleft}\end{minipage}}\addtocounter{nootje}{1}}
\begin{document}

\title[Rigid surfaces]{On rigid compact complex surfaces and manifolds  }

\author{Ingrid Bauer}
\address{Mathematisches Institut,
         Universit\"at Bayreuth,
         95440 Bayreuth, Germany.}
\email{Ingrid.Bauer@uni-bayreuth.de}

\author{Fabrizio Catanese}
\address{Mathematisches Institut,
         Universit\"at Bayreuth,
         95440 Bayreuth, Germany.}
\email{Fabrizio.Catanese@uni-bayreuth.de}
\date{\today}
\thanks{
The present work took place in the  framework  of  
the ERC-2013-Advanced Grant - 340258- TADMICAMT \\
We would like to thank Elisabetta Colombo, Paola Frediani and Alessandro Ghigi for useful conversations; also Philippe Eyssidieux for pointing out 
the reference to Zheng's paper, thus preventing the second author and Alessandro Ghigi from duplicating existing work.
}

 \makeatletter
    \def\@pnumwidth{2em}
  \makeatother

\maketitle

\addtocontents{toc}{\protect\setcounter{tocdepth}{1}}

\tableofcontents

\section{Introduction}

The present investigation originated from  some natural questions concerning the series  of  families of surfaces exhibited in \cite{cat-dett2}
to provide  counterexamples to a question posed by Fujita in 1982 
(see also \cite{cat-dett1}, \cite{cat-dettCR}).  The families depend on some integer invariants,  the main one being an arbitrary integer  $n$ coprime with 6,
and for $n=5$ they were first constructed in \cite{volmax}:  we  shall refer to them as BCD-surfaces.

In case $n=5$, the surfaces  $S$ are ball quotients, hence they possess a K\"ahler metric with strongly negative curvature tensor,
their universal covering $\tilde{S}$  is diffeomorphic to the Euclidean space $\RR^4$, $\tilde{S}$ is a Stein manifold, and the surfaces $S$ are rigid in all possible senses
(see the definitions given in section one).

It is natural to ask whether similar properties hold for the other BCD surfaces, in particular to ask about their rigidity.
In fact, we prove the following

\begin{theorem}\label{bcd}
The BCD surfaces are infinitesimally rigid and rigid.  As a consequence, since  their Albanese map is a semistable fibration $\alpha : S \ra B$ onto a curve $B$ of genus $b : = \frac{1}{2} (n-1)$,
and with fibres of genus $g = (n-1)$,
we get  a rigid curve $B$  inside the moduli stack $\overline{\mathfrak M_{n-1}}$ of stable curves of genus $(n-1)$.

\end{theorem}

An interesting feature of the fibration is that all fibres are smooth, except three fibres which are the union of two smooth curves of genus $b$
intersecting transversally in one point, so that the Jacobians of all the fibres are principally polarized Abelian varieties,
and $B$ yields a complete curve inside $\mathfrak A_{n-1}$; it is an interesting question whether  this curve  is   also rigid  inside $\mathfrak A_{n-1}$ (see \cite{moonen}, \cite{cfg} and
\cite{fgp} for related questions).

Suspicion of rigidity came from the observation that all deformations of BCD also have an Albanese map which is a fibration onto a curve of genus $b$
with exactly three singular fibres, of the same type as described above.   The proof of rigidity however follows another path: first of all  we observe that BCD surfaces admit
as a finite unramified covering the Hirzebruch-Kummer coverings $HK_{CQ}(n)$, the minimal resolution of  a covering  of the plane with Galois group $(\ZZ/n)^5$, and branched on a complete quadrangle $CQ$
 ($CQ$ is the union of the six lines in $\PP^2$ joining four points in linear general position).
 Then one observes easily (proposition \ref{etale}) that if we have a finite  unramified $ Y \ra X$, and $Y$  is rigid, then a fortiori $X$ is rigid too.
 
 Hence theorem \ref{bcd} is implied by the following stronger

\begin{theorem}\label{hk}
The Hirzebruch Kummer surfaces $HK_{CQ}(n)$  are infinitesimally rigid and rigid for all $n \in \NN, n \geq 4$.  

\end{theorem}

The proof of the above theorem occupies the main body of the paper, is mainly based first on the fact that our surfaces are finite Galois coverings 
of the Del Pezzo surface of degree 5, which is  the blow up of $\PP^2$ in four points in general linear position, and which is the moduli space for ordered 5-tuples of  points in $\PP^1$, and as such it admits a biregular  action 
by the symmetric group $\mathfrak S_5$. The other ingredients are  Pardini 's formulae for direct image of sheaves under Abelian coverings,
and then residue sequences associated to sheaves of  logarithmic differential forms: these lead to difficult calculations which can be handled
using symmetry (by the semidirect product of $(\ZZ/n)^5$ with $\mathfrak S_5$) and the very explicit descriptions of the Picard group of the Del Pezzo surface.

Afterwards, it became only natural to put this result in perspective: what do we know in general about rigid complex surfaces,
and about rigid compact complex manifolds?

For curves, the answer is easy: the only rigid curve is $\PP^1$ , so Kodaira dimension $0,1$ is excluded.

For complex surfaces, we can use the Enriques-Kodaira classification to show  that again Kodaira dimension $0,1$ is excluded, and more precisely
we have

\begin{theorem}\label{surfaces}
Let $S$ be a smooth compact complex  surface, which is  rigid. Then either
\begin{enumerate}
\item $S$ is a minimal surface of general type, or
\item $S$ is a Del Pezzo surface of degree $d \geq 5$ ( i.e., one of the following surfaces: $\PP^2 $,  $\PP^1 \times \PP^1 = \FF_0$,
$\FF_1 = S_8$,  $S_7, S_6, S_5$; where $S_{9-r}$ is the blow-up of $\PP^2 $
in $r$ points which are in general linear position).
\item $S$ is an Inoue surface of type $S_M$ or $S_{N,p,q,r}^{(-)}$ (cf. \cite{inoue}).
\end{enumerate}
 Surfaces in classes (2) and  (3) are  infinitesimally rigid. Rigid surfaces in class (1) are also globally rigid,
as well as those in (3),  but the only rigid surface in class (2)  is the projective plane $\PP^2$.
\end{theorem}

 Hence, as  already observed,  for surfaces rigidity implies that the Kodaira dimension is either $ - \infty$ or maximal, equal to $2$
($S$ is of general type), so that here Kodaira dimension $\kod = 0,1$ is excluded. 

We   then show that a similar phenomenon is not true in higher dimension $ n \geq 3$

\begin{theorem}\label{kod}
For each $n \geq 3$ and for each $k = - \infty, 0,2, \dots n$ there is a rigid projective variety $X$ of dimension $n$ and Kodaira dimension $\kod (X) =k$. 

\end{theorem}

The construction of these examples (the case $n=3, k=0$ is due to Beauville) is not so difficult, since essentially rigidity is preserved by products
and by rigid unramified quotients. It seems likely that the exception $\kod=1$ should not occur, at least for $n$ large; but we postpone the answer 
to this question to a  future time.

Global rigidity for rigid varieties of general type is a consequence of the existence of moduli spaces, while in the $ \kod =  - \infty$ case it  becomes
rather complicate for $n \geq 3$, as shown by the work of Siu \cite{siu1}, \cite{siu2}, 
 Hwang \cite{hwang} and Hwang-Mok \cite{hm}, essentially only the case of $\PP^n$
and of the hyperquadric $Q^n$ being solved.

Theorem \ref{surfaces} shows therefore that the problem of classifying rigid  surfaces reduces to the same question for surfaces of general type.
Here our new examples add to a not so long list:

\begin{enumerate}
\item  {\em ball quotients}: for these the  universal covering is the two-dimensional complex ball
$\BB_2 \subset \CC^2$, and  they are pluri-rigid (\cite{siu0}, \cite{mostow}), i.e.  rigid in any possible way, as it happens for the
\item {irreducible bi-disk quotients}: for these the 
universal covering of $S$ is $\BB_1\times \BB_1 \cong \HH \times \HH$, where $\HH$ is the upper half plane,
and  the fundamental group $\pi_1(S) = \Gamma$ has dense image for any 
of the two projections $\Gamma \rightarrow \PSL(2, \RR)$(\cite{jy}, \cite{mok}).
\item {\em Beauville surfaces}: these are the rigid unramified quotients of products of curves (\cite{isog}). They are  infinitesimally rigid, strongly rigid  but not  \'etale rigid,
which means that they have a finite unramified covering which is not rigid.
\item {\em Mostow-Siu surfaces}, \cite{m-s}; these are pluririgid, since they have a metric with strongly negative  curvature.
\item some {\em Kodaira fibrations}  constructed by Catanese-Rollenske \cite{cat-roll}.

\end{enumerate}

All these examples have in common the feature that their universal covering is diffeomorphic to $\RR^4$, so they are classifying spaces $K(\pi,1)$ of some finitely
generated group $\pi$. 

The previous  observations lead to two quite interesting  questions:

\begin{question} \
\begin{itemize}
\item[A)] Does there exist an infinitesimally rigid surface of general type which is not a $K(\pi,1)$ ?

\item[B)] Does there exist a rigid, but not  infinitesimally rigid surface of general type?
\end{itemize}
\end{question}

For question A), by the Lefschetz hyperplane section theorem, it would suffice to find a rigid ample divisor in a rigid threefold which is a $K(\pi,1)$.
For question B), a natural approach, due to the result of Burns and Wahl \cite{b-w},  would be to find a minimal surface of general type $S$
whose canonical model is singular and rigid.

We pose now several  questions,   hoping that  the readers will find them  interesting.

The subject of arrangements of lines in $\PP^2$ has attracted a lot of attention of algebraic geometers after the work of Hirzebruch (\cite{hirz}) 
which provided explicit examples of ball quotients as Hirzebruch-Kummer coverings branched on rigid arrangements of lines
(among these, the most famous are, beyond the complete quadrangle, the Hesse configuration $(9_4, 12_3)$ of 12 lines joining 
pairs of flexpoints of a smooth cubic curve, and its dual configuration $(12_3, 9_4)$ of the 9 lines dual to the flexpoints).

Natural questions are:
\begin{question} \
\begin{itemize}
\item[I)] For which rigid configuration $\sC$ of lines in $\PP^2$ is the associated Hirzebruch Kummer covering $ HK_{\sC}(n)$ rigid for $ n > > 0$? 

\item[II)] For which rigid configuration $\sC$ of lines in $\PP^2$ is the associated Hirzebruch Kummer covering $ HK_{\sC}(n)$ a $K(\pi,1)$ for $ n > > 0$? 

\item[III)] For which rigid configuration $\sC$ of lines in $\PP^2$ does the associated Hirzebruch Kummer covering $ HK_{\sC}(n)$ possess a K\"ahler metric
of negative sectional curvature $ n > > 0$? 

\item[IV)] For which rigid configuration $\sC$ of lines in $\PP^2$ is the associated Hirzebruch Kummer covering $ HK_{\sC}(n)$ \'etale rigid
for  $ n > > 0$? 
\end{itemize}
\end{question}

Observe that if III) has a positive answer, then also II), by the Cartan-Hadamard theorem. Moreover existence of a strongly negative metric
(\cite{siu0}, \cite{m-s}) implies \'etale rigidity.

For the surfaces $HK_{CQ}(n)$ the answer to II), III) and to strong and \'etale rigidity follows, in case that $5$ divides $n$, from the work of Fangyang Zheng \cite{zheng}
who extended the Mostow-Siu technique  to the case of normal crosisings. The case of other integers $n \geq 4$ is open.

Panov \cite{panov} asserts (without giving full details, hence without specifying explicitly the meaning of $ n > > 0$) a positive answer to II) for the surfaces $HK_{CQ}(n)$ and other examples by Hirzebruch: his method consists in finding
polyhedral metrics of negative curvature. So the following question is not yet settled:

\begin{question}
  Are  the surfaces $HK_{CQ}(n)$,  for $ n  \geq 5$, $K(\pi,1)$ spaces  (or even for $n \geq 4$)? 
 
\end{question}

 Of course one could ask a similar question also for non rigid configurations.  
 
 For rigid configurations the philosophy  that for $n >> 0$
the deformations of $ HK_{\sC}(n)$ should correspond to the ones of the configuration  is based on the following partly heuristic argument (a weaker result, i.e. up to taking the product with a smooth manifold, 
was used by Vakil in \cite{murphy} for some special configurations).

Assume that a point of the configuration $P$ has valency $v_P \geq 3$ (i.e., at least 3 lines of the configuration pass through $P$: then the point $P$ has to be blown up
and projection from $P$ induces on  $ HK_{\sC}(n)$, for $ n \geq 4$, a fibration over a curve $B_P$ of genus $\geq 2$. The existence of this fibration is a topological property of
the surface $S$ which is the minimal resolution of the singular Abelian covering of the plane; hence this fibration is stable under deformation, and
any deformation embeds in a product of generalized Fermat curves. Also the number of singularities on the fibres of each such fibration is a topological invariant, and if the  components of singular fibres
are stable by deformation (this is true for BCD surfaces, since there is only one non separating vanishing cycle), then  the exceptional curves would be
stable under deformation, and the question would be reduced to proving that the  equisingular deformations of the finite $(\ZZ/n)^r$ covering of $\PP^2$ are trivial.

The middle step seems to be the most difficult one, and that's why in this paper we are obliged to a rather computationally involved proof; another reason for this
is that the easy criterion given by Pardini (corollary 5.1 ii) of \cite{pardini}) does not apply, since it is easy to show that, $Y$ being the Del Pezzo surface of degree 5,
there are plenty of characters $\chi$ of $G = (\ZZ/n)^5$ for which $H^1 (\Theta_Y (- L_{\chi}) \neq 0$. Hence proving that all deformations are natural
is  a question of the same order of difficulty of proving rigidity.

We feel somehow that our results are like the tip of the iceberg, and to illustrate this philosophy we describe in the last section a new series of rigid line configurations,
which is in some way the most natural construction (possibly known outside of algebraic geometry?): we call this the nth iterated Campedelli Burniat 
configuration $\sC_{CB} (n)$, since for $n=1$ it was used  by Campedelli, and later by Burniat. For $n=0$ $\sC_{CB} (n)$ is the complete quadrangle,
and in the iterative steps we apply a contraction sending the `external triangle' to the one with vertices the midpoints of the sides.

For many  other known configurations, we defer to Hirzebruch's summary \cite{izv} of Hofer's thesis, and to the book \cite{bhh}.

This vast material offers ample source of examples in order to test the above questions in many concrete cases.


\section{Rigidity}

We start recalling the basic  notions of {\em rigidity} for compact complex manifolds $X$ of complex dimension $n$.

\begin{definition}\label{rigid} \
\begin{enumerate}
\item  Two compact complex manifolds $X$ and $X'$ are said to be {\em  deformation equivalent} if and only if there is a 
proper smooth holomorphic map  $$f \colon \mathfrak{X}  \rightarrow \sB 
$$ 
where $\sB$ is a connected (possibly not reduced) complex space and there are points $b_0, b_0' \in \sB$ such that the fibres $X_{b_0} : = f^{-1} (b_0),
X_{b'_0} : = f^{-1} (b'_0)$ are respectively isomorphic to $X, X'$ ($X_{b_0} \cong X, X_{b'_0} \cong X'$).

\item  Two compact complex manifolds $X$ and $X'$ are said to be {\em direct deformation of each other} if and only if there is a 
proper smooth holomorphic map  $$f \colon \mathfrak{X}  \rightarrow \sB
$$ 
as in (1), but where moreover $\sB$ is assumed to be irreducible. 
\item
Equivalently, two compact complex manifolds $X$ and $X'$ are  {\em direct deformation of each other} if and only if there is a 
proper smooth holomorphic map  $$f \colon \mathfrak{X} \rightarrow  \Delta
$$ 
where   $\Delta \subset \CC$ is the unit disk, and where $X$, respectively $ X'$, are isomorphic to fibres of $f$.
\item  Equivalently, deformation equivalence is the equivalence relation generated by the relation of  direct deformation. This means
that two compact complex manifolds $X$ and $X'$ are  {\em deformation equivalent} if and only if there is a sequence of compact complex manifolds $(X_{i})_{i \in \{0,1, \ldots, k\}}$ such that $X_0 = X$, $X_k=X'$ and $X_{i}$ is a direct deformation of $X_{i-1}$. 
\item A compact complex manifold $X$ is said to be {\em globally rigid} if for any compact complex manifold $X'$, which is deformation equivalent to $X$, we have  an isomorphism $X \cong X'$.
\item  A compact complex manifold $X$ is instead said to be  {\em (locally) rigid} (or just {\em rigid}) if for each deformation of $X$,
$$f \colon (\mathfrak{X},X)  \rightarrow (\sB, b_0)
$$ 
there is an open neighbourhood $U \subset \sB$ of $b_0$ such that $X_t := f^{-1}(t) \cong X$ for all $t \in U$.
\item  A compact complex manifold $X$ is said to be  {\em infinitesimally rigid} if 
$$H^1(X, \Theta_X) = 0,$$
where $\Theta_X$ is the sheaf of holomorphic vector fields on $X$.
\item
$X$ is said to be  {\em strongly rigid} if the set of compact complex manifolds $Y$ which are homotopically equivalent to $X$,
$\{ Y |  Y  \sim_{h.e.} X\}$ consists of a finite set of isomorphism classes of globally rigid varieties.
\item
$X$ is said to be  {\em \'etale  rigid} if every \'etale (finite unramified) cover $Y$ of $X$ is rigid (we can obviously combine this concept with the previous ones, 
and speak of  \'etale  globally rigid, \'etale infinitesimally  rigid, ...
\end{enumerate}
\end{definition}

\begin{remark}
1) If $X$ is infinitesimally rigid, then $X$ is also locally rigid. This follows by the Kodaira-Spencer-Kuranishi theory, since $H^1(X, \Theta_X)$ is the Zariski tangent space of the germ of analytic space which is the base $\Def(X)$ of the Kuranishi semiuniversal deformation of $X$.
So, if  $H^1(X, \Theta_X) =0$, $\Def(X)$ is a reduced point and all deformations are induced by the trivial deformation.
In other words, the condition of  infinitesimal rigidity is equivalent to  the condition that every deformation of $X$, when restricted to a suitable
neighbourhood $U$ of $b_0$, be isomorphic to the trivial deformation $ X \times U$.

2) More generally, the  definitions are so given that obviously strong rigidity implies global rigidity, which in turn implies local rigidity,
as well as \'etale  rigidity implies local rigidity.

3) The Fischer-Grauert theorem (\cite{gf})says conversely that if $\sB$ is reduced, then the condition of local rigidity yields
triviality of the family over a suitable neighbourhood of $b_0$.

4) Moreover, the Kuranishi theorem (\cite{kur1}, \cite{kur2}) implies that the number of moduli of $X$, defined as $ m(X) : = \dim \Def(X) $,
satisfies $$ m(X) = \dim \Def(X)  \geq h^1(X, \Theta_X) - h^2(X, \Theta_X).$$
Hence, if $\Def(X)$ is reduced, and $m(X) \geq 1$,  then necessarily $X$ is not locally rigid.
More generally,  if $\Def(X)$ is reduced, the Kuranishi family is universal 
if $ h^0(X, \Theta_X) = 0 $ or $ h^0(X, \Theta_{X_t})  $
is a locally constant function for $ t \in  \Def(X)  $ ,  \cite{wavrik}.

5) For $n : = \dim X =1$, all the  notions of rigidity are  equivalent and it is well known that the only rigid curve is $\PP^1$.

6) The following well known examples (see \cite{ravello}) illustrate the difference between global and infinitesimal rigidity.
The Segre-Hirzebruch surface $\FF_n : = \PP (\hol_{\PP^1} \oplus \hol_{\PP^1} (n))$
has a smooth Kuranishi space which is the germ at the origin of the vector space 
$$  \Ext^1 ( \hol_{\PP^1} (n),  \hol_{\PP^1}) \cong \CC^{n-1} \  {\rm for } \ n \geq 1, \ = 0  \ {\rm for } \ n=0.$$
The family parametrizes extensions
$$ 0 \ra  \hol_{\PP^1} \ra V \ra   \hol_{\PP^1} (n) \ra 0,$$
and the surfaces in the deformation are all the surfaces of the form  $\FF_{n-2k}$, $n \geq 2k$.

Hence $\PP^1 \times \PP^1 = \FF_0$ and $ \FF_1$ are infinitesimally rigid, but not globally rigid.

\end{remark}

The following is a useful general result:

\begin{theorem}\label{m(S)}
A  compact complex manifold $X$ is rigid, if and only if  the Kuranishi space $\Def(X)$ (base of the Kuranishi family of deformations)
is  $0$-dimensional. 

In particular, if $X=S$ is a smooth compact complex surface and 
$$
10 \chi(\mathcal{O}_S) - 2 K_S^2 + h^0(X, \Theta_S) >  0,
$$
then $S$ is not rigid.

\end{theorem}

\begin{proof}
The `if' part being obvious by the versality of the Kuranishi family, we show the `only if' part.

Without loss of generality, we can assume that for all $ t \in \Def (X)$ $X_t \cong X$. Let $B \subset \Def(X)$ be the reduced subspace
$ B : = \Def(X)_{red}$. Then, by the theorem of Fischer- Grauert \cite{gf} it follows that the pull back of the Kuranishi family to $B$ is trivial,
isomorphic then to $ B \times X$.

If we assume that $B$ is not a point, then there is $t \in B$ such that  the derivative of the inclusion map $ i : B \ra \Def(X)$
is non zero. Hence the Kuranishi family is not semiuniversal in the point $ i(t)$, contradicting Corollary  1 of \cite{meersseman1},
which asserts that if $ h^0 (X_t, \Theta_{X_t})$ is constant, then the Kuranishi family is semiuniversal (versal in the author's unusual terminology)
at each point.

In the case of surfaces, we use the Kuranishi inequality and Riemann-Roch:

\begin{multline}\label{ineq}
\dim \Def(S) \geq   h^1( \Theta_S) - h^2(\Theta_S) =  - \chi(\Theta_S) + h^0(S, \Theta_S)  = \\
= 10 \chi(\mathcal{O}_S) - 2 K_S^2 + h^0(X, \Theta_S).
\end{multline}

\end{proof}

Now, before we dwell  upon the analysis of rigidity in complex dimension $2$, let us make a few easy but important observations.
First of all, if $X$ or $Y$ are not rigid, then the same holds for the product $X \times Y$.  We can moreover say when is 
$X \times Y$ infinitesimally rigid.

\begin{prop}
Let $X$ and $Y$ be infinitesimally rigid compact complex manifolds.  Then $X \times Y$ is  infinitesimally rigid
if and only if 
$$ h^0(X, \Theta_X) \cdot h^1 (Y, \hol_Y) =  h^0(Y, \Theta_Y)  \cdot h^1 (X, \hol_X)  = 0 .$$
\end{prop}
\Proof
The result is an easy consequence of the K\"unneth formula by which
$$ H^1 (  X \times Y, \Theta_{X \times Y}) \cong $$
$$ \cong   H^1(X, \Theta_X) \oplus  (H^0(X, \Theta_X) \otimes  H^1 (Y, \hol_Y))
\oplus (H^1 (X, \hol_X) \otimes  H^0(Y, \Theta_Y)) \oplus H^1(Y, \Theta_Y).$$ 

\qed

\begin{prop}\label{etale}
If $p : Z \ra X$ is \'etale, i.e. a finite unramified holomorphic map between compact complex manifolds,
 then  the  infinitesimal  rigidity of $Z$ implies the  infinitesimal  rigidity of 
$X$.
Moreover,  if $Z$ is rigid, then also $X$ is rigid.

\end{prop}
\Proof
For the first assertion, simply observe that $  H^1(Z, \Theta_Z) =  H^1(X, p_* (\Theta_Z) ) = 0$,
and that $p_* (\Theta_Z) = p_* (p^* \Theta_X)) =  \Theta_X \otimes (p_* \hol_Z)$
has $\Theta_X $ as a direct summand.

Proof of the second assertion: we have seen in theorem \ref{m(S)} that rigidity is equivalent to the condition that $\Def(X)$ is zero dimensional.

If $X$ is not rigid,  we pass to an unramified covering $W$ of $Z$, which is a Galois cover of $X$.

In this way there is a subgroup $ H \subset G$ such that $ Z = W / H$, while $X = W / G$.

We use now a result from \cite{montecatini}: 
$$  \Def (X) = \Def(W)^G , \  \Def (Z) = \Def(W)^H \Rightarrow \Def (X) \subset \Def (Z) .$$ 

We conclude: if  $ \Def (Z)$ has dimension $0$, i.e. it is a (possibly non reduced) point, a fortiori  $\Def (X)$
is a point.

\qed
\begin{remark}
1) The preceding propositions show that in dimension strictly higher than $2$ it is easy to construct many examples
of infinitesimally rigid varieties by taking \'etale quotients of products of infinitesimally rigid varieties.
More generally, if $G$ is a finite group acting on $X, Y$ in such a way that  both actions are rigid
(this is a weaker notion than the rigidity of $X,Y$), then necessarily the quotient $W:   (X \times Y)/G$
is infinitesimally rigid ($W$ is a manifold if the diagonal action of $G$ on $X \times Y$ is free).

2) The example of Beauville surfaces, the rigid quotients $ (X \times Y)/G$, where $X$ and $Y$ are 
curves of respective genera greater or equal to $2$ shows that the converse to proposition \ref{etale}
does not hold. Beauville surfaces are rigid, globally rigid, strongly rigid \cite{isogenous}, but not \'etale rigid.

\end{remark}

\begin{theorem}
Let $S$ be a smooth compact complex  surface, which is (locally) rigid. Then either
\begin{enumerate}
\item $S$ is a minimal surface of general type, or
\item $S$ is a Del Pezzo surface of degree $d \geq 5$, $\PP^2 $ or $\PP^1 \times \PP^1 = \FF_0$,
$\FF_1 = S_8$, or $S_7, S_6, S_5$; where $S_{9-r}$ is the blow-up of $\PP^2 $
in $r$ points which are in general linear position.
\item $S$ is an Inoue surface of type $S_M$ or $S_{N,p,q,r}^{(-)}$ (cf. \cite{inoue}).
\item
Rigid surfaces in class (1) are also globally rigid, surfaces in class  (3) are  infinitesimally and globally rigid,
 surfaces in class (2) are infinitesimally rigid, but the only rigid surface in class (2)  is the projective plane $\PP^2$.
\end{enumerate}
\end{theorem}

\begin{remark}
For surfaces of general type it is expected  to find examples which are  rigid,
but not infinitesimally rigid: such an  example would be the one of  a minimal surface $S$ such that its canonical model $X(S)$ is infinitesimally rigid
and singular (see \cite{b-w}).

\end{remark}

\begin{proof}

Let us begin with the last statement, (4).

1) A locally rigid minimal surface of general type is also globally rigid, due to the existence of a global moduli space for surfaces of general type
\cite{gieseker}.

2) A Del Pezzo surface $S$ of degree $d \geq 5$ is infinitesimally rigid, but $S$ is globally rigid if and only if $S \cong \PP^2$.
This was already observed for $\FF_0, \FF_1$.  For $S_{9-r}$ with $ r \geq 2$ it suffices to move the second point
of the blow up until it becomes infinitesimally near to the first, so that we get $Y$ with $- K_Y$  not ample. Since $-K_S$ is ample,
$Y$ is a deformation of $S$ which is not isomorphic to $S$. 

 3) The infinitesimal rigidity of the Inoue surfaces in (3) was shown in \cite{inoue}.
For their global rigidity, observe that the Inoue surfaces, belonging to three classes,  $S_{N,p,q,r}^{(-)}$,  $S_M$ or $S_{N,p,q,r}^{(+)}$,
are characterized (\cite{inoue}, \cite{teleman})  by the condition that  $b_1(S)=1, b_2(S) = 0$ and that they contain no curves; 
the condition $b_1(S)=1, b_2(S) = 0$ is stable by deformation and singles out also the
minimal  Hopf surfaces. The fundamental group $\pi_1(S)$  is also invariant by deformation; it suffices then to show that the fundamental groups
of surfaces in class (3) cannot be equal to the fundamental groups of a Hopf surface or of an Inoue surface of type $S_{N,p,q,r}^{(+)}$.
This is easy for the case of Hopf surfaces, which have a finite unramified covering which is a primary Hopf surface, diffeomorphic to $S^3 \times S^1$:
hence for a Hopf surface $\pi_1(S)$ has cohomological dimension $1$ ($\Ga$ has cohomological dimension $n$ if $H^i (\Ga, \QQ) = 0 \ \forall i > n$). While the universal covering of a Inoue surface is $\HH \times \CC$, hence for an Inoue surface $H^i (\pi_1(S), \QQ) = H^i (S, \QQ) $ and $\pi_1(S)$ has cohomological dimension $4$.

In the case of $S_{N,p,q,r}^{(+)}$ it is quicker to use a deformation theoretic argument. In fact for   a surface of type $S_{N,p,q,r}^{(+)}$ $H^i (S, \Theta) = \CC$ for $i=0, i=1$,
and $H^2 (S, \Theta) = 0$ (implying that the Kuranishi family has a basis $\Def (S)$ smooth of dimension $1$ containing only surfaces of the same type), while for surfaces in (3)  $H^i (S, \Theta) = \CC$ for $i=0,1,2$: hence,  
by semicontinuity of the dimension of $H^i (S, \Theta) = \CC$ for $i=0, i=1$, an Inoue surface as in (3) cannot be a limit of a family of Inoue surfaces
of type $S_{N,p,q,r}^{(+)}$.

Let us now proceed with the proof of the main statements, (1)-(3).

a) Let $S$ be a rigid smooth compact complex surface and assume that $S$ is not minimal. Let 
$$
S =S_0 \rightarrow S_1 \rightarrow \ldots \rightarrow S_k =S'
$$
be a sequence of point blow-ups such that $S'$ is minimal. By Kodaira's theorem (cf. page 86 of \cite{kodaira63}), if $S_i$ is rigid, then also the pair $(S_{i+1}, p_{i+1})$ (where $S_i \rightarrow S_{i+1}$ is the blow-up in $p_{i+1} \in S_{i+1}$) is rigid.

Varying the point $p_{i+1} \in S_{i+1}$ we see that 
$$
\dim \Aut(S') \geq 2k.
$$
In particular there are two linearly independent global holomorphic vector fields on $S'$ and we get an exact sequence
$$
0 \rightarrow \mathcal{O}_{S'}(D_1) \oplus \mathcal{O}_{S'}(D_2) \rightarrow \Theta_{S'} \rightarrow \mathcal{F} \rightarrow 0,
$$
where $D_1, D_2 \geq 0$ are effective divisors on $S'$ and $\mathcal{F}$ is a coherent sheaf with $\dim \Supp(\mathcal{F}) \leq 1$.

Taking the determinant, this implies that $-K_{S'} \geq D_1+D_2$. Therefore either
\begin{itemize}
\item[i)] $- K_{S'}$ is a strictly effective divisor, or
\item[ii)] $K_{S'} \equiv 0$.
\end{itemize}

Case i) bifurcates:

\begin{itemize}
\item
(i-1) if $S'$ is algebraic, then $S'$ is ruled (since $K_{S'}H<0$ for $H$ very ample)
\item
(i-2)
 if $S'$ is not algebraic, then $S'$ is a surface of class $VII_0$, in particular $b_1(S') = 1, q(S') = 1, p_g(S')= 0 \Rightarrow \chi (S') =0, b_2(S') = e (S') =  -K^2_{S'}.$
\end{itemize}
In case (ii) $S'$ is either a K3-surface, or a complex torus,  or a Kodaira surface. But all these surfaces are not rigid:  K3 surfaces  have $\chi(S')= 2$
hence we can apply theorem \ref{m(S)}.  That  tori and Kodaira surfaces are not rigid is well known (for Kodaira surfaces, look at \cite{kodaira}).

In case   (i-2) again theorem \ref{m(S)} applies as soon as $b_2(S') > 0$, since $10 \chi(S') - 2 K^2_{S'} = 2 b_2(S') > 0$.

If instead $S'$ is a surface of class $VII_0$ with $b_2(S') = 0$, rigidity for $S$ (again by  theorem \ref{m(S)} ) implies that $k = 0$.
Finally, if $k=0$ then $S = S'$ and  $S$ is a Hopf surface or an Inoue surface (\cite{teleman}). By \cite{inoue} an Inoue surface is rigid if and only if it is of type $S_M$ or of type$S_{N,p,q,r}^{(-)}$.

Hopf surfaces are not rigid by \cite{kodaira2}, \cite{kato} and  \cite{dabrowski}.

Thus we have seen that if $S$ rigid and non minimal then each of its   minimal models $S'$ is ruled.

Now, a minimal ruled surface $S'$ is either $\PP^2$ or is a $\PP^1$-bundle over a curve $C$ of genus $g=g(C)$. 

\underline{Claim.} $S'$ must be regular.

\underline{Proof of the claim:} Observe that a minimal ruled surface $S'$ has 
$$
K_{S'}^2 = 8(1-g), \ \ \chi(\hol_{S'}) = 1-g.
$$
This implies that
$$0 = \dim \Def(S) \geq 6g-6 + 2k.
$$
 Therefore either $g=0$ or $g=1$ and $k=0$. 
 
 There are now two ways to proceed.  The first argument   uses  \cite{seiler}, Lemma 2 and Lemma 3 which says  that $h^0(\Theta_{S'}) \geq 1$
 for $g=1$ and $k=0$: this is a contradiction  since then  $ \dim \Def(S') \geq h^0(\Theta_{S'}) \geq 1$.
 
 The second argument is more geometric: every ruled surface is obtained from the product $ \PP^1 \times C$ via a sequence of elementary transformations.
 Now, if the genus $g$ of $C$ is $\geq 1$, then $C$ is not rigid, moreover we can perform all the required elementary transformations that
 lead to $S$ when we deform $C$: hence we can arbitrarily deform the Albanese variety $C$ of $S$ and $S$ is not rigid.
 
 \qed
 
 Therefore $S'$ is either $\PP^2$ or a $\PP^1$-bundle over $\PP^1$, in particular $\chi(S') = \chi(S) = 1$. 
 Since $S$ is assumed to be rigid, this implies that $10 \leq 2K_S^2$, i.e. $K_S^2 \geq 5$.

\underline{Claim.} $S$ is a Del Pezzo surface, i.e., $-K_S$ is very ample.

But this follows since if one blows up special points in $\PP^1 \times \PP^1$ or in $\PP^2$, this contradicts the rigidity of $S$. 

b) We can now assume $S$ to be minimal. 

If $\kod(S) = - \infty$ and $S$ is not ruled, then $S$ is of type $VII_0$,  a case which we already treated.

Hence  the theorem will be proven once we can rule out the case where $S$ is minimal of  Kodaira dimension 0 or 1.
Assume $S$ to be minimal with $\kod(S) =0,1$. Then $K_S^2 = 0$ and $\chi(S) \geq 0$. Since $10\chi(S) -2 K_S^2 = 10 \chi(S) \leq 0$ we have $\chi(S) = 0$. This rules out, as done earlier,  $K3$-surfaces and Enriques surfaces. Moreover, we have already observed that Abelian surfaces and Kodaira surfaces 
are not rigid; the same holds for  hyperelliptic surfaces.

Hence we are reduced to the case $\kod(S) = 1$, i.e. properly elliptic surfaces.

Assume that $\kod(S) = 1$ and consider for a suitable $m >>0$ the $m$-th canonical map 
$$
\varphi = \varphi_{|mK_S|} \colon S \rightarrow B.
$$
By the formula of Zeuthen-Segre, since $\chi(S) = 0 \Rightarrow e(S) = 0,$ we see that the singular fibres are multiples of smooth elliptic curves, and $\varphi$ is isotrivial. Then there exists a finite Galois covering $B' \ra B$ such that the normalization $S'$ of the fibre product  $ B' \times_B S$ is isomorphic to a product $ B' \times E$, where
$E$ is an elliptic curve. 

Now $S = S' /G$, where $G$ acts on $B'$ with quotient $B$, and $G$ acts on $E$ by translations (since all the fibres of $\varphi $
have no rational component). This shows that we can freely deform the elliptic curve $E$, hence $S$ is not rigid.

\end{proof}

\section{Examples}

 \begin{remark}
The up to now known examples of rigid surfaces of general type are the following:
\begin{enumerate}
\item the so-called {\em ball quotients}: these are the smooth projective surfaces $S$ whose universal covering is the two-dimensional complex ball
$\BB_2$; they are infinitesimally rigid, strongly rigid  and \'etale rigid.
\item {irreducible bi-disk quotients}, i.e. the 
universal covering of $S$ is $\BB_1\times \BB_1 \cong \HH \times \HH$, where $\HH$ is the upper half plane,
moreover if we write $S = \HH \times \HH / \Gamma$ the fundamental group $\Gamma$ has dense image for any 
of the two projections $\Gamma \rightarrow \PSL(2, \RR)$; they are infinitesimally rigid, strongly rigid  and \'etale rigid.
\item {\em Beauville surfaces}; they are  infinitesimally rigid, strongly rigid  but not  \'etale rigid.
\item {\em Mostow-Siu surfaces}, \cite{m-s}; these are infinitesimally rigid, strongly rigid  and \'etale rigid.
\item some {\em Kodaira fibrations} (Catanese-Rollenske) \cite{cat-roll}; these are rigid and strongly rigid, infinitesimal rigidity
and  \'etale rigidity is not proven in \cite{cat-roll} but could be true.
\end{enumerate}
\end{remark}

\begin{definition} A compact complex manifold is called
  a {\em projective classifying space} if its universal covering $\tilde{X}$ is contractible.
\end{definition}

\begin{remark}
1) All known examples of rigid surfaces of general type are projective classifying spaces.

2) The examples (1)-(3), and (5) are strongly rigid.
\end{remark}

In the previous section we have seen that if $n = \dim (X) =2$ and $X$ is rigid, then either the Kodaira dimension of $X$ is $ - \infty$,
or $ Kod (X) = n$ ($X$ is of general type).

It is an interesting question whether we can say anything about the Kodaira dimension of rigid manifolds of a given dimension $n$.

Next, we show that we can obtain, for $n \geq 4$, rigid generalized hyperelliptic manifolds (they have Kodaira dimension $ \kod (X) = 0$).

\begin{theorem}\label{kod0}
Let $E$ be the Fermat (equianharmonic) elliptic curve, with the standard action of $G : = (\ZZ/3)^2$:
$$E : = \{ (x:y:z) \in \PP^2 | x^3 + y^3 + z^3 = 0 \},$$
let 
$$ e_1 (x:y:z) := (\e x:y:z), \  e_2 (x:y:z) := ( x:\e y:z), \  \e : = \exp (\frac{2}{3} \pi i ).$$
Define $e_3 : =  - e_1 - e_2, e_4 : = e_1 - e_2.$

Consider the following automorphisms of $G$, defined uniquely by the conditions:
\begin{itemize}
\item $\psi_1 (e_1) : = e_1$,  $\psi_1 (e_2) : = e_2$,
\item $\psi_2 (e_1) : = e_1$, $\psi_2 (e_2) : = - e_2$,
\item $\psi_3 (e_1) : = e_4$, $\psi_3 (e_2) : =  e_3$, 
\item $\psi_i (e_1) : = e_1$, $\psi_i (e_2) : = - e_4$ for $i \geq 4$, and
\end{itemize}
 let $G$ act on $E^n$ by:
$$  g (p_1, \dots, p_n) : = (  \psi_1 (g) (p_1),   \psi_2 (g) (p_2), \dots,   \psi_n (g) (p_n)).$$

Then $G$ acts on $E^n$ freely for $n \geq 4$, and $X : = (E^n)/ G $  is a   rigid compact complex manifold with
Kodaira dimension equal to zero, strongly  rigid exactly  for $n=4$.
\end{theorem}
\begin{proof}
Observe that the elements in $G$ which have fixed points on $E$ are just the multiples of the vectors $e_1,e_2, e_3 : =  - e_1 - e_2$,
whereas the non zero multiples of $e_4$ act freely.

It suffices to show that the action on $E^4$ is free, because then the action on $E^n$ is a fortiori free for $ n \geq 4$.

If $g \in G \setminus \{ 0\}$ is a  multiple of $e_1$, then $\psi_3(g)$ is a multiple of $e_4$, hence $g$ acts freely;
  (nontrivial) multiples of $e_4$ act freely since
 $\psi_1(e_4) = e_4$,   for multiples of $e_2$ this follows from   $\psi_4(e_2) = e_4$,  and for multiples of $e_3$ since  $\psi_2(e_3) = - e_4$.

Hence the action is free and $X$ is a complex manifold.

 We claim that $X$ is infinitesimally rigid. For this, let $Z:= E_1 \times \ldots \times E_n$, where $E_i \cong E$ for all $i$, and $G$ acts via $\psi_i(g)$ on $E_i$.

Then by the K\"unneth formula we have
$$H^1 (\Theta_Z) = \Big( \bigoplus_{i =1}^{n} H^1 (\Theta_{E_i}) \Big)  \oplus \Big( \bigoplus_{i=1} ^n  \Big(  H^0 (\Theta_{E_i}) \otimes (\oplus_{j\neq i}
H^1(\hol_{E_j}))  \Big)\Big).$$
 
We have
$$
H^1 (\Theta_X) = H^1 (\Theta_Z)^G = \Big( \bigoplus_{i =1}^{n} H^1 (\Theta_{E_i})^G \Big) \oplus \Big( \bigoplus_{i=1} ^n  \Big(  H^0 (\Theta_{E_i}) \otimes (\oplus_{j\neq i}
H^1(\hol_{E_j}))  \Big)^G \Big).
$$
Now, $H^1 (\Theta_{E_i})^G = 0 $, while for the other terms we define
 $\varphi$ to be the character $e_1^*$; it is the character of the representation on $H^0(\Omega^1_{E_1})$.
 
Then the character of $H^0 (\Theta_{E_i}) \otimes 
H^1(\hol_{E_j})$ equals $  - \varphi \circ \psi_i - \varphi \circ \psi_j$; it is therefore nontrivial if the character $\varphi \circ \psi_i + \varphi \circ \psi_j$
is nontrivial, which follows by our assumption. In fact the characters in question are just:

$$
\varphi \circ \psi_1 \begin{pmatrix} m_1 \\ m_2 \end{pmatrix} = m_1, \ \varphi \circ \psi_2 \begin{pmatrix} m_1 \\ m_2 \end{pmatrix} = m_2, \ \varphi \circ \psi_3 \begin{pmatrix} m_1 \\ m_2 \end{pmatrix} = m_1-m_2,
$$
and
$$
\varphi \circ \psi_i \begin{pmatrix} m_1 \\ m_2 \end{pmatrix} = m_1 + m_2, \ \forall i \geq 4.
$$

Therefore $H^1 (\Theta_X) =0$, whence $X$ is infinitesimally rigid.

Moreover, any $Y$ homotopically equivalent to $X$ has an unramified $G$-cover $W$ which has the same integral cohomology algebra
of a complex torus, hence (see \cite{topm} for this and other assertions) is a complex torus with an action of the group $G$. 

Let $\Lambda : = \pi_1(W) \cong \pi_1(E^n)$: the $G$-action on $\Lam$  is induced by conjugation via the exact sequence 
$$ 1 \ra  \Lambda \ra   \pi_1(Y) \cong \pi_1(X) \ra G\ra 1.$$

The action of $G$ on $\Lambda$  is a direct sum $\Lambda = \oplus_1^4 \Lambda_i$,
where $\Lambda_i$ is a free $R_3 : = \ZZ[x] / (x^2 + x +1)$-module, of  rank respectively 1,1,1,$n-3$,
 on which $G$ acts via the surjection $ \ZZ[G] \ra R_3$ corresponding to the character $ \phi_i : G  \ra \ZZ/3$
with kernel generated by $\psi_i^{-1} (e_4)$. 

Since for each $i \neq j \in \{1,2,3,4\}$ there is a $g \in G$ such that $\psi_i(g)$ has fixed points on $E$, whereas $\psi_j(g)$ acts freely,
we see (as in  \cite{bacainoue}, or \cite{topm} page 389) that $W$ splits as a direct sum $W_1 \times W_2 \times W_3 \times W_4$
 of three elliptic curves  and one torus, corresponding to the complex subspaces $ \Lambda_i \otimes \RR$ of the complex vector space $ \Lambda \otimes \RR$.
 
 Because the $W_i$ are stable for the $G$ action, and $\Lambda_i$ is a free $R_3 $- module, it follows that $W_i$ is the Fermat elliptic curve
 for $i=1,2,3$, and also for $i=4$ if $n=4$. 
 
 In the case where $ n > 4$, we can take a family of complex structures on $\Lambda_4 \otimes \CC = H_1 \oplus H_2$
 (here $H_1$ is the eigenspace for the character $\phi_4$, $H_2$ is the eigenspace for the character $\overline{\phi_4}$,), such that  $H^{1,0}$ is the direct sum of an arbitrary  k-dimensional subspace 
 of $H_1$ with an arbitrary  subspace of $H_2$ of dimension $n-3-k$.
 
  For $ 1 < k < n-3$ we obtain a family of non rigid manifolds, since then $$H^1 (\Theta_W)^G = (( H^{1,0})^{\vee} \otimes \overline{H^{1,0}} )^G
   \supset  \Hom ( H_1,\overline{H_2})^G  \neq 0.$$
 \end{proof}

In the following we show that for $n\geq 3$ there are examples of   rigid projective manifolds $X$ of Kodaira dimension  $2, \ldots ,n$.

\begin{theorem} \label{kod>1}

 For $n \geq 3$ there are rigid $n$-dimensional manifolds of Kodaira dimension $k$, for each $k$ with $2 \leq k \leq n$.

\end{theorem}

 \begin{proof}
 Consider the group 
 $$
 G := <x,y,z,w,t| x^3,y^3,z^3,w^3,t^3,y^x =yz,z^x =zw, z^y =zt>.
 $$
 
 This group has order $3^5$ and has abelianization $G/[G,G] = (\ZZ/3\ZZ)^2$. More precisely, $G$ sits inside an exact sequence 
 $$
 0 \rightarrow (\ZZ/3\ZZ)^3 \rightarrow G \rightarrow G/[G,G] = (\ZZ/3\ZZ)^2 \rightarrow 0,
 $$
 
 where $\varphi \colon G \rightarrow G/[G,G]$ satisfies $\varphi(x) = (1,0) = e_1$, $\varphi(y) = (0,1)=e_2$ and $\varphi(z) =\varphi(w)=\varphi(t)=0$.
 
 By \cite{barbosfair} $G$ admits a Beauville structure of type $(3,3,9), (3,3,9)$ given by $(x,y,(xy)^{-1})$, $(xt,y^2w,(xty^2w)^{-1})$, i.e., we get two triangle curves $\lambda_i \colon C_i \rightarrow C_i/G \cong \PP^1$.

Let $S= (C_1 \times C_2)/G$ be the above described Beauville surface and let $E$ be  the Fermat elliptic curve with the action  defined in theorem \ref{kod0} and consider e.g. the (non faithful) $G$-action 
 $$
 g(x) := \varphi(g)(x)
 $$
 on $E$. Then obviously the  diagonal action of $G$ on $C_1 \times C_2 \times (C_2)^{k-2} \times E^{n-k}$ is free, and the quotient is a projective manifold of 
  dimension $n$ and Kodaira dimension $k$.
We claim that $X$ is infinitesimally rigid. The proof is the same as in the proof of the infinitesimal rigidity in theorem \ref{kod0}.
 \end{proof}
  \begin{remark}
 The above manifolds are   strongly rigid for $n-k=1$, as it follows by the same proof as in  theorem \ref{kod0}.
 Twisting the action on the Fermat elliptic curves as in loc. cit. we can show strong rigidity for $n-k \leq 4.$
 
 \end{remark}

 \begin{remark}
   That there do exist rigid threefolds of Kodaira dimension $0$
 was shown by Beauville: in  \cite{katatabea},  page 5, he constructed a rigid Calabi-Yau $m$-fold, $m=3,4,6$,
 as the minimal resolution of a quotient $E^m  / (\ZZ/m)$ where $\ZZ/m$ acts on $E$
 by multiplication $ z \ra \eta z$, $\eta$ being a primitive $m$th root of unity, and diagonally on $E^m$.
  \end{remark}

\section{Results related to the geometry of the  Del Pezzo surface of degree $5$} \label{delpezzo}

Consider the smooth Del Pezzo surface $Y$ of degree $5$, which  is the blow up $Y \colon = \hat{\PP}^2(p_1, p_2,p_3,p_4)$ of the plane in  four points  $p_1, p_2, p_3, p_4 \in \PP^2$ in general position.

$Y$ contains exactly  10 lines which are in bijection with (unordered) pairs $\{i,j\} \subset \{1,2,3,4,5\}$ in the following classical way:
\begin{itemize}
\item $E_{i5}$ is the exceptional curve over $p_i$, $i \in \{1,2,3,4\}$;
\item $E_{ij}$ is the strict transform of the line in $\PP^2$ passing through $p_h$ and $p_k$, where $\{i,j,h,k\} = \{1,2,3,4\}$.
\end{itemize}

Note that by a slight abuse of notation we may also write $E_{ij}$ with $ i >j$ instead of $E_{\{i,j\}}$, just assuming silently that $E_{ij} = E_{ji}$.

The following is the intersection behaviour of the 10 lines.  We have,  for $i,j,h,k \in \{1,2,3,4,5\}$:

$$
E_{ij}E_{hk} = \left\{ \begin{array}{rl} -1, & \mbox{if} \ \  | \{i,j\} \cap \{h,k\} | = 2, \\
0, & \mbox{if} \ \ | \{i,j\} \cap \{h,k\} | = 1, \\
1, & \mbox{if} \ \  | \{i,j\} \cap \{h,k\} | = 0.
\end{array} \right.
$$

\begin{remark} \

1) The symmetric group $\mathfrak{S}_5$ acts on $Y$ by permutation of the indices $\{1, \ldots , 5\}$. This action induces a permutation of the five pencils on $Y$
which yield  conic bundle structures:
\begin{itemize}
\item $X_i:=|L-p_i|$, $i \in \{1,2,3,4\}$, whose general element is  the strict transform of a line in $\PP^2$ passing through $p_i$;
\item $X_5 := |2L -p_1-p_2-p_3-p_4|$, whose elements are the strict transforms of the conics passing through $p_1,p_2,p_3,p_4$.
\end{itemize}
2) Each of the five pencils $X_i$ has three reducible fibers corresponding to the three $(2,2)$ partitions of $\{1,2,3,4,5\} \setminus \{i\}$; e.g. for $i=1$:
$$
L-p_1 \equiv E_{34} + E_{25} \equiv E_{24} + E_{35} \equiv E_{23} + E_{45}.
$$
3) $\Pic(Y) \cong \ZZ^5$ is generated by the ten lines $E_{ij}$, $1 \leq i < j \leq 5$ and the relations are generated by the linear equivalences
associated to the pencils: for each subset $I:= \{i_1,i_2,i_3,i_4\}$ of cardinality 4 of $\{1,2,3,4,5\}$ we have
$$
E_{i_1i_2} + E_{i_3i_4} \equiv E_{j_1j_2} + E_{j_3j_4}, \ \ \mbox{for} \ \ \{j_1,j_2,j_3,j_4\} = I.
$$
\end{remark}

In the sequel we are going to study the linear independence in $\Pic(Y)$ of a subset of  the set of 10 lines $E_T$ for $T =\{t_1,t_2\} \subset \{1,2,3,4,5\}$.
In fact, we prove the following:
\begin{proposition}\label{dependencies}
Let $T_1, \ldots , T_k \subset \{1,2,3,4,5\}$ be $k$ pairwise different subsets of cardinality 2. Then we have:

\begin{enumerate}
\item  $\rk \langle E_{T_1}, \ldots , E_{T_6} \rangle < 5  \iff  $ either $ \exists \ j \in \{1,2,3,4,5\}$ s.th. $T_1, \ldots T_6 \subset \{1,2,3,4,5\} \setminus \{j\}$, or there exist $ i \neq j$ such that $\{T_1, \ldots , T_6\} $ is the set of pairs which intersect the subset  $\{i,j\}$ in exactly one element.
In  both cases: $\rk \langle E_{T_1}, \ldots , E_{T_6} \rangle =4$;
\item $\rk \langle E_{T_1}, \ldots , E_{T_5} \rangle =4 \iff $ we have two reducible fibres of a pencil $X_j$
$\iff  \exists \  j \in \{1,2,3,4,5\}$ s.th (after maybe renumbering the $T_i$'s) $j \notin T_1, T_2, T_3, T_4$ and $T_1 \cup T_2 = T_3 \cup T_4 =   \{1,2,3,4,5\} \setminus \{j\}$;
\item  $\rk \langle E_{T_1}, \ldots , E_{T_k} \rangle = 5$, if $k \geq 7$.
\end{enumerate}
\end{proposition} 

\begin{proof}
We start the proof with a list of elementary observations.
\begin{remark}\label{remarks}  \

i) If $T_1, \ldots T_6 \subset \{1,2,3,4,5\} \setminus \{j\}$, then $E_{T_1}, \ldots , E_{T_6}$ are the six irreducible components of the three reducible fibres of the pencil $X_j$ hence by Zariski's lemma  these six irreducible curves generate a subgroup of rank 4 in $\Pic(Y)$. Moreover, for any $T_7 \neq T_1, \ldots, T_6$, we have $\rk \langle E_{T_1}, \ldots , E_{T_7} \rangle = 5$ since $T_7$ has non zero intersection with the fibre of $X_j$.

ii) If there is a $\ j \in \{1,2,3,4,5\}$ s.th. $j \in T_1, T_2, T_3, T_4$, then $E_{T_1}, E_{T_2}, E_{T_3}, E_{T_4}$ are pairwise disjoint and linearly independent in $\Pic(Y)$. If $T_5 \neq T_1, T_2, T_3, T_4$, then $\rk \langle E_{T_1}, \ldots , E_{T_5} \rangle =5$, since the intersection matrix
$$
\begin{pmatrix}
-1&0&0&0&1 \\
0&-1&0&0&1 \\
0&0&-1&0&1 \\
0&0&0&-1&1 \\
1&1&1&1&-1 
\end{pmatrix}
$$
has rank 5.

iii) If $T_1,T_2,T_3, T_4 \subset  \{1,2,3,4,5\} \setminus \{j\}$, then we have two cases:
\begin{itemize}
\item[a)]  $\rk \langle E_{T_1}, \ldots , E_{T_4} \rangle =3$ $\iff$ none of the four curves $E_{T_i}$ is disjoint from the others $\iff$ 
we have the components of two reducible fibres of the pencil $X_j$
$\iff$ for each $T = \{i_1,i_2\} \in \{T_1, T_2, T_3, T_4\}$, also $T' = \{1,2,3,4,5\} \setminus \{j, i_1, i_2\} \in \{T_1, T_2, T_3, T_4\}$;
\item[b)]  $\rk \langle E_{T_1}, \ldots , E_{T_4} \rangle =4$ $\iff$ there is an $i \in  \{1,2,3,4\}$ (equivalently, there are two such indices $i$)
such that  $E_{T_i}$ is disjoint from the others.
\end{itemize}
In case b): if $T_5$ is not a component of a reducible fibre of the pencil $X_j$, then $\rk \langle E_{T_1}, \ldots , E_{T_5} \rangle =5$.
\end{remark}

Observe that i) shows one direction of (1).
Assume now that there are $T_1, \ldots , T_6 \subset \{1,2,3,4,5\}$ such that $\rk \langle E_{T_1}, \ldots , E_{T_6} \rangle < 5$: then by ii) we know that for each  $j \in \{1,2,3,4,5\}$,  $j$ is not contained in four of the $T_i$' s. We can clearly assume that we are not in case i), hence for each $j \in \{1,2,3,4,5\}$ there is an $ i \in \{1,2,3,4,5,6\}$ such that $j \in T_i$.
An easy counting argument shows that there is a $ j \in \{1,2,3,4,5\}$ such that $j \in T_{i_1}, T_{i_2}, T_{i_3}$,
for three distinct such $T_i$'s.

Without loss of generality we can assume that $j=1$ and
$$
T_1 = \{1,2\}, \ T_2 = \{1,3\}, \ T_3 = \{1,4\}.
$$
Moreover, since there is an $i$ such that $5 \in T_i$, we can assume that $T_4 = \{2,5\}$. Observe now that $\rk \langle E_{T_1}, \ldots , E_{T_4} \rangle =4$ and we have a section $T_3$ of $X_4$, one reducible fibre $T_2 + T_4$, and anothere component of a reducible fibre.

If  $T_5$ is a component of the third reducible fibre ($T_5 = \{ 1,5\}$ or $T_5 = \{ 2,3\}$), then we get $\rk \langle E_{T_1}, \ldots , E_{T_5} \rangle = 5$,  whereas if $T_5 = \{ 3,5\}$ the rank remains $4$, similarly if  $T_5 = \{ 4,5\}$ since then we have two reducible fibres of the pencil $X_3$.

  We get $\rk \langle E_{T_1}, \ldots , E_{T_5} \rangle = 5$ for $T_5 = \{ 2,4 \}$ (consider the pencil $X_3$), and for  $T_5 = \{ 3,4 \}$ 
  (intersecting  with
  $E_{12}$ and $E_{34}$ we see that if $E_{24}$ were linearly dependent of the other four, we would have 
  $$ E_{24} = - E_{12 }  + 2 E_{14} + a E_{13} + b  E_{25}, $$
  contradicting that the intersection number of  $E_{24}$ with a fibre of $X_4$ is $1$.

The conclusion is that  $\rk \langle E_{T_1}, \ldots , E_{T_6} \rangle = 5$
unless there is $j \notin T_i \ \forall i=1, \dots, 6$  or unless, up to symmetry, we have the six curves 
$$
T_1 = \{1,2\}, \ T_2 = \{1,3\}, \ T_3 = \{1,4\}, T_4 = \{2,5\}, \ T_5 = \{3,5\}, \ T_6 = \{4,5\}.
$$

This shows (1) and (3).
 
In order to show (2), observe that by ii)  each  $j \in \{1,2,3,4,5\}$   is  contained in at most three of the $T_i$'s,
and it at least one, else we have the reducible fibres of a pencil $X_j$, and we are done. 
We can therefore either assume that we are again in the situation above
$$
T_1 = \{1,2\}, \ T_2 = \{1,3\}, \ T_3 = \{1,4\} , \ T_4 = \{2,5\},
$$
or each $j$ belongs to exactly two $T_i$' s.

In the former case  $\rk \langle E_{T_1}, \ldots , E_{T_5} \rangle =5$ for any choice of $T_5$ different from $\{3,5\}, \{4,5\}$, 
or we have two reducible fibres of the pencil $X_4$, respectively of the pencil $X_3$, which is our
assertion.

 In the latter case
without loss of generality we can assume that 
$$
T_1 = \{1,2\}, \ T_2 = \{2,3\}, \ T_3 = \{3,4\}, \ T_4 = \{4,5\}, \ T_5 = \{5,1\},
$$
i.e., we have a pentagon and the intersection matrix has rank $5$.

\end{proof}

Consider
$$
D \colon = \bigcup_{\{i,j\} \subset  \{1,2,3,4,5\}} E_{ij} \subset Y
$$
and let $e_1, e_2, e_3, e_4, e_5$ be a basis of $(\ZZ / n \ZZ)^5$.

\begin{definition}
Given a normal complex space $Y$ and a closed analytic subspace  $D$ containing $Sing (Y)$,
the {\em  Kummer covering of exponent $n$} of $Y$ branched on $D$
(also called the {\em maximal Abelian covering of exponent $n$} of $Y$ branched on $D$) is the
ramified locally  finite covering  $\pi : X \ra Y$ with $X$  a  normal analytic space 
such that the   restriction of $\pi$  to $ Y \setminus D$ is the Galois unramified covering
  associated to the monodromy homomorphism
  
  $$
\varphi \colon H_1(Y \setminus D, \ZZ) \rightarrow H_1(Y \setminus D, \ZZ) \otimes \ZZ/n \ZZ.
$$
 In the case where $Y$ is the Del Pezzo surface of degree 5, and $D$ is the union of the 10 lines of $Y$,
 we shall speak of the {\em Hirzebruch-Kummer covering of exponent $n$ associated to the complete quadrangle.}

 In this special case
$$
\varphi \colon H_1(Y \setminus D, \ZZ) \rightarrow H_1(Y \setminus D, \ZZ) \otimes \ZZ/n \ZZ \cong H_1(Y \setminus D, \ZZ / n \ZZ)  \cong (\ZZ / n \ZZ)^5, 
$$

Observe that $H_1(Y \setminus D, \ZZ) \cong \ZZ^5$ is generated by 10 elements  $\epsilon_{ij} =  \epsilon_{ji}$
where $\epsilon_{ij}$ is the class of a small loop around the line $E_{ij}$.

 The relations satisfied by these generators are the linear combinations of the $\mathfrak S_5$-transforms of the relation
 \begin{equation}\label{relation}
 \epsilon_{12}  = \epsilon_{34} + \epsilon_{35} +\epsilon_{45}.
 \end{equation}

In terms of these generators, the homomorphism $\varphi$ is concretely 
given by 
$$
\begin{array}{lll}
 \epsilon_{13} \mapsto e_5, & \epsilon_{14} \mapsto e_1, & \epsilon_{23} \mapsto e_4,\\
 \epsilon_{24} \mapsto e_2, & \epsilon_{34} \mapsto e_3, \\
\end{array}
$$

\end{definition}

\begin{remark}
Then 
\begin{itemize}
\item $\epsilon_{12} \mapsto -(e_1 + \ldots + e_5)$;
\item $\epsilon_{45} \mapsto  -(e_1 + e_2 + e_3)$;
\item $\epsilon_{15} \mapsto  e_2 + e_3 + e_4$;
\item $\epsilon_{25} \mapsto e_1 + e_3 + e_5$;
\item $\epsilon_{35} \mapsto -(e_3 + e_4 + e_5)$.
\end{itemize}
\end{remark}

\begin{remark}
For $\sigma \in (\ZZ / n \ZZ)^5$  denote by $D_{\sigma}$ the union of the components of $D$ having $ \sigma$ as local monodromy.

In our case  either $D_{\sigma} = 0$ or $D_{\sigma} $ is irreducible, and consists of the unique curve $E_{ij}$
such that $\varphi (\e_{ij}) = \s$.
\end{remark}

 \begin{remark}
The Hirzebruch-Kummer covering of exponent $n$ associated to the complete quadrangle shall here  for brevity be referred to
as the $HK(n)$-surface: it has a group of automorphisms which is the semidirect group of $G : =  (\ZZ / n \ZZ)^5$ with $\mathfrak S_5$.

In the case $n=2$ we obtain a K3 surface, which was investigated by van Geemen, van Luijk and others \cite{vanvan};
the case $n=3$ was investigated by Roulleau in \cite{roulleau}.
\end{remark}

We shall prove now the following:
\begin{proposition}\label{rankcond}
Consider for $n \geq 4$,  the  Kummer covering of exponent $n$ of $Y$ branched in $D$. Let $\psi \in ((\ZZ / n \ZZ)^5)^*$ be a character of $G:=  (\ZZ / n \ZZ)^5$. 

1) If $n \neq 4,6$, then
$$
\rk\{D_{\sigma} : \psi(\sigma) \neq -1\} <5 \iff \{D_{\sigma} : \psi(\sigma) \neq -1\} =\{A, B_1, B_2,B_3\},
$$
where $A, B_i$ are irreducible components of $D$ with $AB_i = 1$, $B_iB_j =0$ for $i \neq j$.

2) If $n=6$, then $\rk\{D_{\sigma} : \psi(\sigma) \neq -1\} <5$ if and only if 
\begin{itemize}
\item $\{D_{\sigma} : \psi(\sigma) \neq -1\} =\{A, B_1, B_2,B_3\}$, where $A, B_i$ are irreducible components of $D$ with $AB_i = 1$, $B_iB_j =0$ for $i \neq j$, or 
\item $\{D_{\sigma} : \psi(\sigma) \neq -1\}=\{E_{ij} : j \in \{1,2,3,4,5\} \setminus \{i\}\}$, for some $i$.
\end{itemize}
3) If $n=4$, then $\rk\{D_{\sigma} : \psi(\sigma) \neq -1\} <5$ if and only if one of the following  conditions is satisfied:
\begin{itemize}
\item $\{D_{\sigma} : \psi(\sigma) \neq -1\}=\{E_{T_1}, \ldots , E_{T_5}\}$, where  there is a $j$ such that $T_i \in \{1, \ldots, 5\} \setminus \{j\}$;
\item $\{D_{\sigma} : \psi(\sigma) \neq -1\}=\{E_{T_1},E_{T_2}  , E_{T_3}\}$, where there are $i,j$ such that $T_1, T_2, T_3 \subset \{1, \ldots, 5\} \setminus \{i,j\}$.
\end{itemize}
\end{proposition}

\begin{proof}
1) Let $\psi = (a_1,\ldots,a_5) \in ((\ZZ / n \ZZ)^5)^*$ be a character of $G:=  (\ZZ / n \ZZ)^5$, such that $\rk\{D_{\sigma} : \psi(\sigma) \neq -1\} \leq 4$.  Then by Proposition \ref{dependencies} we know that 
$$
|\{D_{\sigma} : \psi(\sigma) \neq -1\}| \leq 6, 
$$
and
\begin{itemize}
\item[a)] if $|\{D_{\sigma} : \psi(\sigma) \neq -1\}| = 6$, then $\{D_{\sigma} : \psi(\sigma) \neq -1\}=\{E_{T_1}, \ldots , E_{T_6}\}$, where 
either

\begin{itemize}
\item[a.1)]  $T_1, \ldots T_6 \subset \{1,2,3,4,5\} \setminus \{j\}$ for some $j$;
\item[a.2)] there are $1 \leq i < j \leq 5$ such that the $T_h$' s are the subsets of cardinality two with $|T_h \cap \{i,j\}|=1$
\end{itemize}

\item[b)] if $|\{D_{\sigma} : \psi(\sigma) \neq -1\}| = 5$, then $\{D_{\sigma} : \psi(\sigma) \neq -1\}=\{E_{T_1}, \ldots , E_{T_5}\}$, where there is a $j$ such that 
(after possibly renumbering the $T_h$'s) $$T_1 \cup T_2 =T_3 \cup T_4 = \{1, \ldots, 5\} \setminus \{j\}.$$
\end{itemize}
We consider the following three cases:
\begin{itemize}
\item[I)] $|\{D_{\sigma} : \psi(\sigma) \neq -1\}| \leq 4$;
\item[II)] $|\{D_{\sigma} : \psi(\sigma) \neq -1\}| = 5$;
\item[III)]  $|\{D_{\sigma} : \psi(\sigma) \neq -1\}| = 6$.
\end{itemize}

{\bf I)} Assume that $\{D_{\sigma} : \psi(\sigma) = -1\} = \{E_{T_1}, \ldots, E_{T_k}\}$ with $k \geq 6$. Then we distinguish two cases:

\underline{Case 1:} there is a $j \in \{1, \ldots, 5\}$ such that $T_1, \ldots , T_6 \subset \{1, \ldots, 5\} \setminus \{j\}$. Without loss of generality we can assume that $j=5$, hence 
$$
\{E_{T_1}, \ldots, E_{T_6}\} = \{E_{12}, E_{13}, E_{14}, E_{23}, E_{24}, E_{34}\}.
$$
$\psi(E_{T_i}) = -1$ implies that $a_1 = \ldots = a_5 = -1$, whence 
$$
-1 = \psi(E_{12}) = -(a_1 + \ldots +a_5) =5,
$$
and this leads to a contradiction for $n\neq 6$.

\underline{Case 2:} for each $j \in \{1, \ldots, 5\}$ there is an $i \in \{1, \ldots, 6\}$ such that $j \in T_i$. 

2-1) Assume here  that there is a $j$ which is contained in four $T_i$'s; then, without loss of generality, we can assume $j=1$ and $T_1 = \{1,2\}$, $T_2 = \{1,3\}$, $T_3 = \{1,4\}$, $T_4 = \{1,5\}$. But then $\psi(E_{T_i})=-1$ implies $a_5=a_1=a_2+a_3+a_4 =-1$ and we get for $n\neq 4$ a contradiction since
$$
-1 = \psi(E_{12}) = -(a_1 + \ldots +a_5) = 3.
$$

2-2) Assume that each $j$ belongs to at least one and at most three $T_h$'s. 
Therefore, again by an easy counting argument, we see that there is a $ j \in \{1,2,3,4,5\}$ such that $j \in T_{i_1}, T_{i_2}, T_{i_3}$. Without loss of generality we may assume that $j=3$ and $T_1 = \{1,3\}$, $T_2 = \{2,3\}$, $T_3 = \{3,4\}$ and since there is a $k$ such that $5 \in T_k$ we may assume that $T_4 = \{4,5\}$. Then $\psi(E_{T_i})=-1$ for $i=1,2,3,4$ implies $a_5=a_4=a_3=-(a_1+a_2+a_3) =-1$. Using now that $[a_1+a_2]=[1-a_3] =2$, hence $a_2 =[2-a_1]$ we obtain:
$$
\begin{array}{ll}
 \psi(E_{12}) = [-(a_1 + \ldots +a_5)] =  1 \neq -1, & \psi(E_{14}) =a_1,\\
\psi(E_{15}) = [a_2+a_3+a_4] = [-a_1], & \psi(E_{24}) = a_2 = [2-a_1], \\
\psi(E_{25}) = [a_1+a_3+a_5] = [-a_2] =[a_1-2], & \psi(E_{35}) = [-(a_3+a_4+a_5)] = 3.\\
\end{array}
$$
In order that two of the above five values are equal to $-1$ (the first is never  $-1$), since we have two pairs of opposite values,
the only possibility is that $a_1= \pm1$.  Wlog we may assume $a_1 = 1$. Then $\psi(E_{15}) = \psi(E_{25}) =-1$, $\psi(E_{14}) = \psi(E_{24}) =1$ and
when $n \neq 4$ 
$$
\{D_{\sigma} : \psi(\sigma) \neq -1\} = \{E_{12},E_{14},E_{24},E_{35}\},
$$
which is exactly the situation described in 1).

{\bf II)} $\{D_{\sigma} : \psi(\sigma) \neq -1\} = \{E_{T_1}, \ldots, E_{T_5}\}$ and we assume $\rk \langle E_{T_1}, \ldots , E_{T_5} \rangle <5$. Then four of the $E_{T_i}$'s have to be as in  Remark \ref{remarks} iii), a), yielding two reducible fibres, and we can assume $T_1 = \{1,2\}$, $T_2 = \{4,5\}$, $T_3 = \{1,4\}$, $T_4 = \{2,5\}$. We have then two possibilities,
according to a section or a  component of the third reducible fibre:

a) $T_5 = \{3,5\}$, or

b)  $T_5 = \{1,5\}$.
In both cases $\psi(E_T) =-1$ for $T \in \{\{1,3\},\{2,3\},\{2,4\},\{3,4\}\}$  implies that $a_5 = a_4 = a_2=a_3 =-1$. This  shows that a) is impossible, since  $\psi(E_{15}) = [a_2+a_3+a_4] = -3\neq -1$; for $n\neq 4$ also b) is impossible, since $-1 = \psi(E_{35}) = [-(a_3+a_4+a_5)] = 3\neq -1$.

{\bf III)}  $\{D_{\sigma} : \psi(\sigma) \neq -1\} = \{E_{T_1}, \ldots, E_{T_6}\}$ and we assume $\rk \langle E_{T_1}, \ldots , E_{T_6} \rangle <5$. Therefore by Proposition \ref{dependencies}, (1), we may assume that either (up to $\mathfrak S_5$-symmetry)
we have
$$
\{D_{\sigma} : \psi(\sigma) = -1\}=\{E_{23}, E_{34},E_{24},E_{15}\},
$$
or
$$
\{D_{\sigma} : \psi(\sigma) = -1\}=\{E_{13}, E_{34},E_{23},E_{35}\}.
$$
The first alternative cannot occur, since then $a_2 = a_3 = a_4 = -1$ hence $a_2 +  a_3 + a_4 = -3 \neq -1$.
The second alternative can only occur for $n = 4$, since we would have $-1 = a_5 = a_3=a_4 = -(a_3+a_4+a_5) = 3$.

Hence assertion  1) is proven.

2) For $n=6$  the only additional case is I), case 1: here $\psi =(5,5,5,5,5)$ and we get $\{D_{\sigma} : \psi(\sigma) =  -1\}=\{E_{ij} : j \in \{1,2,3,4,5\} \setminus \{5\}\}$. This proves 2).

3)  For $n=4$ we need to treat several possibilities:
\begin{itemize}
\item
I) Case 2-1
\item
I) Case 2-2
\item
II)  possibility b
\item
III) second alternative.

\end{itemize}

\underline{I) Case 2-1:}  $T_1 = \{1,2\}$, $T_2 = \{1,3\}$, $T_3 = \{1,4\}$, $T_4 = \{1,5\}$. 

 $\psi(E_{T_i})=-1$ implies $a_5=a_1=a_2+a_3+a_4 =-1$ and  $-1 = \psi(E_{12}) = -(a_1 + \ldots +a_5) = 3$. Then we get the values for the components of the 3 reducible fibres:
\begin{itemize}
\item $\psi(\e_{25}) = -2+a_3$,  $\psi(\e_{34}) = a_3$,
\item $\psi(\e_{35}) = -2+a_2$,  $\psi(\e_{24}) = a_2$,
\item $\psi(\e_{45}) = -2+a_4$,  $\psi(\e_{23}) = a_4$.
\end{itemize}

Since  two of the above values must be $= -1$,
we can wlog assume that $a_2=a_3=1$,  hence also $a_4 =1$ and we obtain the second exceptional case in the statement 3)
$$
\{D_{\sigma} : \psi(\sigma) \neq -1\}= \{E_{23}, E_{24}, E_{34} \}.
$$

\underline{I) Case 2-2:}  the previous analysis shows that for $n=4$
$$
\{D_{\sigma} : \psi(\sigma) \neq -1\} = \{E_{12},E_{14},E_{24}\}.
$$

\underline{II) b:} here
$\{D_{\sigma} : \psi(\sigma) = -1\}=\{E_{13}, E_{23},E_{24},E_{34}, E_{35}\}$.

 $\psi(E_T) =-1$ for $T \in \{\{1,3\},\{2,3\},\{2,4\},\{3,4\}\}$  implies that $a_5 = a_4 = a_2=a_3 =-1$, then we have also $\psi(E_{35}) = [-(a_3+a_4+a_5)] = 3 = -1$. Since we assume that $|\{D_{\sigma} : \psi(\sigma) = -1\}| = 5$, and the other values are $1, a_1, 2-a_1, a_1-2, -a_1$ 

 we get 
$\psi=(a_1,-1,-1,-1,-1)$ with $a_1 = 0,2$. In both cases 
$$
\{D_{\sigma} : \psi(\sigma) \neq -1\} = \{E_{12}, E_{14}, E_{15}, E_{25},  E_{45}\}.
$$

\underline{III) second alternative:}   
$$
\{D_{\sigma} : \psi(\sigma) = -1\}=\{E_{13}, E_{34},E_{23},E_{35}\}.
$$
We show now that this cannot occur, since then we would first of all
have  $a_5 = a_3=a_4 = -(a_3+a_4+a_5) = -1$. 

Since moreover we assume that $|\{D_{\sigma} : \psi(\sigma) = -1\}| = 4$) and the other six values
are:
$$ a_1, a_1-2, a_2, a_2 -2, 1 - a_1 - a_2, -1 -a_1 - a_2,$$
we obtain that $a_1, a_2 \neq 1, -1$.
Also, we should have $a_1 + a_2 \neq 0,2$, hence we derive a contradiction.

\end{proof}

\section{Invariants of Hirzebruch-Kummer coverings associated to the \protect\linebreak complete  quadrangle}

The next two sections shall lead to the proof of the following main result:

\begin{theorem}\label{main}
Let $\pi \colon S \rightarrow Y$ be the surface $HK(n)$, the Kummer covering of exponent $n\geq 3$ of the Del Pezzo surface $Y$ of degree 5,
 branched on the  divisor $D \in |-2K_Y|$, union of the 10 lines of $Y$ . Then $S$ is a smooth surface of general type with $K_S$ ample, and 
$$
H^1(S, \Theta_S) = 0
$$
for $ n\geq 4$. Hence $S$ is infinitesimally and globally rigid  for $ n\geq 4$.
\end{theorem}

Consider the usual eigensheaf decomposition
$$\pi_* \hol_S \cong \bigoplus_{\psi \in ((\ZZ/n \ZZ)^5)^*}  \mathcal{L}^{-1}_{\psi} .$$
In the remaining part of the section, for each character $\psi \in ((\ZZ / n \ZZ)^5)^*$ we want to calculate the character sheaves $\sL_{\psi}$
(for $\psi = 0$, $\sL_{0} = \hol_Y$).

\begin{remark}
Viewing  $Y \rightarrow \PP^2$ as the blow up of the plane $\PP^2$  in $p_1,p_2,p_3,p_4$, we denote by $L$ the pull back of a line in $\PP^2$ and by $E_i$  (instead of $E_{i5}$) the exceptional curve over $p_i$. In this way we get a standard basis for $\Pic(Y)$.
\end{remark}

We begin with the following lemma, where we
denote by  $[] $ the remainder after  division by $n$, $[]: \ZZ / n \ZZ \rightarrow \{0, \ldots, n-1\}$, $ a \mapsto [a]$.

\begin{lemma}
For each character $\psi = (a_1, \ldots a_5) \in ((\ZZ/n \ZZ)^5)^*$ we have:
\begin{multline}
n \mathcal{L}_{\psi} \equiv F(a_1, \ldots, a_5) L - \lambda_1(a_1, \ldots , a_5) E_1 - \\ 
-\lambda_2(a_1, \ldots , a_5) E_2 - \lambda_3(a_1, \ldots , a_5) E_3- \lambda_4(a_1, \ldots , a_5) E_4 ,
\end{multline}
where
\begin{itemize}
\item $F(a_1, \ldots, a_5) = [a_1] +[a_2]+[a_3]+[a_4]+[a_5]+[-(a_1+ \ldots + a_5)]$,
\item $\lambda_1(a_1, \ldots , a_5) = [a_2]+[a_3]+[a_4] - [a_2+ a_3 + a_4]$,
\item $\lambda_2(a_1, \ldots , a_5) = [a_1]+[a_3]+[a_5] - [a_1+ a_3 + a_5]$,
\item $\lambda_3(a_1, \ldots , a_5) = [a_1]+[a_2]+[-(a_1+ \ldots + a_5)] - [-(a_3+ a_4 + a_5)]$,
\item $\lambda_4(a_1, \ldots , a_5) = [a_4]+[a_5]+[-(a_1+ \ldots + a_5)] - [-(a_1+ a_2 + a_3)]$,
\end{itemize}
\end{lemma}
\begin{proof}
As in \cite{volmax}, p. 392, we use the formula 
$$
n \mathcal{L}_{\chi} \equiv \sum_{1 \leq i<j\leq5} \psi(\varphi(\e_{ij})) E_{ij}.
$$
The claim follows from writing the right hand side in terms of the basis $L, E_1,E_2,E_3,E_4$ of $\Pic(Y)$.

\end{proof}

\begin{remark} We have the following:
\begin{enumerate}
\item $F(a_1, \ldots, a_5) \equiv 0 \mod n$,  $0 \leq F(a_1, \ldots, a_5) \leq 5n$,  
\item $F(a_1, \ldots, a_5) = 0 \iff (a_1, \ldots, a_5) = 0$,
\item $\lambda_i(a_1, \ldots, a_5) \equiv 0 \mod n$, $0 \leq \lambda_i(a_1, \ldots, a_5) \leq 2n$, 
\item $ \lambda_1 +\lambda_2 +\lambda_3 +\lambda_4 = 2F(a_1, \ldots, a_5) - S(a_1, \ldots, a_5)$,
\end{enumerate}
where 
$$
S(a_1, \ldots, a_5) \colon = [-(a_1+ a_2 + a_3)]+[a_2+ a_3 + a_4] + [a_1+ a_3 + a_5] +  [-(a_3+ a_4 + a_5)].
$$
Then $S(a_1, \ldots, a_5) \equiv 0 \mod n$,  $0 \leq S(a_1, \ldots, a_5) \leq 3n$.

\end{remark}
\begin{remark}\label{lambda2n}
If for $\psi = (a_1, \ldots a_5)$ we have $\lambda_i (a_1, \ldots a_5)= 2n$, then $\psi(E_i) \neq n-1$.
\end{remark}
\begin{proposition}
Let $S$ be the $HK(n)$ surface, and let $\pi \colon S \rightarrow Y$ be the  Kummer covering of exponent $n$ of $Y$ branched in $D$, where $n\geq 3$. Then $S$ is a smooth surface of general type with $K_S$ ample.

Moreover, we have:

$$
K_S^2 = 5(n-2)^2n^3, \ \ e(S) = n^3(2n^2-10n+15).
$$
\end{proposition}
 \begin{proof}
The canonical divisor of $S$ multiplied by $n$  satisfies
 $$ n K_S = \pi^* [ n K_Y + (n-1) D]  = -  \pi^* (n-2) K_Y \Rightarrow K_S^2 = n^3 (n-2)^2 K_Y^2 = n^3 (n-2)^2 5.$$
 In particular $K_S$ is ample for $ n \geq 3$.
 
 We use then the additivity of the topological Euler Poincar\'e characteristic,
 obswrving that the ten lines meet in 15 points, and each intersects three other lines. Hence
 $$  7 = e(Y) = e (Y-D) + e(D^*) + 15,$$
 where $D^*$ is the disjoint union of the ten lines, three times punctured, so that  $e(D^*) = - 10$ and $ e (Y-D)=2$.
 
 Using that unramified coverings of degree $d$ multiply $ e$ by $d$, we get 
 
 $$ e(S) = n^3 [  2 n^2  - 10 n  + 15  ]  .$$ 

\end{proof}

For the proof of theorem \ref{main} we use the following formulae by  R. Pardini (cf. \cite{pardini}).
Recall that $ \Omega_Y^1( \log D_j)_{j \in J}$ is the sheaf of meromorphic 1-forms generated as a sheaf of  $\hol_Y$-modules
by $ \Omega_Y^1$ and by $d(\log \de_j)$, for each divisor $D_j = div (\de_j)$, $j \in J$.

\begin{proposition}
Let $G$ be an Abelian group and let $\pi \colon S \rightarrow Y$ be a Galois cover with group $G$ between compact smooth manifolds. Then for each character $\psi \in G^*$ we have
$$
\pi_*(\Omega_S^1 \otimes \Omega_S^2)^{\psi} = \Omega_Y^1( \log D_{\s} : \s \in \sS_{\psi})(K_Y + \mathcal{L}_{\psi}),
$$
where $\sS_{\psi} \colon = \{\s : \psi (\s)  \neq n - \frac{n}{m(\s)} | m (\s) = \ord (\s)  \}$.
\end{proposition}

\begin{remark}\label{h2=0}
In our situation: $\s \in G$ has always order $ m (\s) = n$, hence 
$$
\sS_{\psi} =  \{\s : \psi (\s)  \neq n - 1  \}.
$$

Recall also once more that each $D_{\sigma} \neq 0$ is an irreducible $(-1)$-curve.
\end{remark}

Observe that by Serre duality we have, for all characters $\psi$: 

$$
h^i (S, \Theta_S)^{- \psi} = h^{2-i}( \Omega_Y^1( \log D_{\sigma} : \psi(\sigma) \neq -1)(K_Y + \mathcal{L}_{\psi})).
$$

Since $S$ is of general type, $h^0 (S, \Theta_S) = 0$, hence in order to show that $h^1 (S, \Theta_S) = 0$ it suffices to prove that $\forall \psi$:
$$
h^0(\Omega_Y^1( \log D_{\sigma} : \psi(\sigma) \neq -1)(K_Y + \mathcal{L}_{\psi})) = \chi (\Omega_Y^1( \log D_{\sigma} : \psi(\sigma) \neq -1)(K_Y + \mathcal{L}_{\psi})).
$$

\begin{lemma}\label{chi}
For each $\psi \in ((\ZZ/ n \ZZ)^5)^*$ and each divisor $\Delta$ on $Y$ we have:
\begin{equation}
\chi (\Omega_Y^1( \log D_{\sigma} : \psi(\sigma) \neq -1)(\Delta)) = \Delta^2 -5 
+ \sum_{\sigma : \psi(\sigma) \neq -1} (1+ D_{\sigma}\cdot \Delta).
\end{equation}
\end{lemma}

\begin{proof}
Consider the exact residue sequence 
\begin{equation}
0 \rightarrow \Omega_Y^1(\Delta) \rightarrow \Omega_Y^1( \log D_{\sigma} : \psi(\sigma) \neq -1)(\Delta) \rightarrow \bigoplus_{\sigma : \psi(\sigma) \neq -1} \mathcal{O}_{D_{\sigma}} (\Delta) \rightarrow 0.
\end{equation}
An easy Chern classes calculation (since $\chi (\Omega_Y^1) = - 5$) yields
$$
\chi(\Omega_Y^1(\Delta)) = \Delta^2 -5
$$ 
moreover 
$$
\chi(\bigoplus_{\sigma : \psi(\sigma) \neq -1} \mathcal{O}_{D_{\sigma}} (\Delta)) = \sum_{\sigma : \psi(\sigma) \neq -1} (1+ D_{\sigma}\Delta)
$$
hence the claim follows.
\end{proof}

\begin{remark}
Recall once more that the branch divisor $D$ of the Kummer covering consists of 10 $(-1)$-curves which have the following property:
\begin{itemize}
\item if $B$ is an irreducible component of $D$, then there exist exactly 3 irreducible components $A_1, A_2, A_3$ of $D$ such that $A_1\cdot B = A_2\cdot B= A_3 \cdot B = 1$, all the other irreducible components of $D$ are disjoint from $B$.
\end{itemize}
\end{remark}

In the sequel we give a complete classification of the classes of the divisors  $K_Y + \mathcal{L}_{\psi}$, $\psi \in ((\ZZ/ n \ZZ)^5)^*$ in the Neron-Severi group $NS (Y)$ of $Y$ (here $ NS(Y) = \Pic (Y)$).

\subsection{{\color{darkgreen}$F(a_1, \ldots , a_5) = n$}} Make first the obvious observation that $F(a_1, \ldots , a_5) $ is greater or equal to the
sum of the positive summands yielding $ \la_i(a_1, \ldots , a_5), \ \forall i.$
If there is an $i \in \{1,2,3,4\}$ such that $\lambda_i = n$, then for the strict transforms of the three lines, say $L_1, L_2, L_3$, passing through $p_i$, we have $\psi(L_1)+\psi(L_2) +\psi(L_3) = n$, hence for the other three lines we have $\psi(L) = 0$ (since $F=n$). Therefore $\lambda_j = 0$ for $j \neq i$. 

Otherwise $\lambda_i=0$ for all $i$.

Therefore in this case we have the possibilities:
\begin{itemize}
\item[(1)] $K_Y + \mathcal{L}_{\psi} \equiv -2L + E_i + E_j + E_k$, where $i,j,k \in \{1,2,3,4\}$ pairwise different;
\item[(2)] $K_Y + \mathcal{L}_{\psi} \equiv -2L  + E_1 + E_2+E_3+E_4$.
\end{itemize}

\subsection{{\color{darkgreen}$F(a_1, \ldots , a_5) = 2n$}}
Assume there is an $i \in \{1,2,3,4\}$ such that $\lambda_i = 2n$, then the same argument as above shows that for the strict transforms of the three lines, say $L_1, L_2, L_3$, passing through $p_i$, we have $\psi(L_1)+ \psi(L_2) + \psi(L_3) = 2n$, while  $\psi(L)=0$ for the strict transforms of the other three lines. Therefore $\lambda_j = 0$ for $j \neq i$ and we get the following possibility:
\begin{itemize}
\item[(3)] $K_Y + \mathcal{L}_{\psi} \equiv -L - E_i + E_j + E_k +E_l$, where $\{i,j,k,l\}=\{1,2,3,4\}$.
\end{itemize}
We can now assume that $\lambda_i \leq n$ for all $i \in \{1,2,3,4\}$. Recall that 
$$
 \lambda_1 +\lambda_2 +\lambda_3 + \lambda_4 = 4n - S, \ \ S \in \{0,n,2n,3n\}.
$$
If $S=0$, then $ \lambda_1 = \lambda_2 =\lambda_3 = \lambda_4= n$, hence we have the possibility:
\begin{itemize}
\item[(4)] $K_Y + \mathcal{L}_{\psi} \equiv -L$.
\end{itemize}

\begin{remark}\label{s=0}
Observe that $S=0$ implies that $a_1 = a_4$, $a_2 = a_5$, $a_3 = [-(a_1+a_2)]$ and $F = 2([a_1] + [a_2]+[a_3])$.
\end{remark}

If $S = n$, then $ \lambda_1 +\lambda_2 +\lambda_3 +\lambda_4 = 3n$, hence there is an $i \in \{1,2,3,4\}$ such that $\lambda_i = 0$ and $\lambda_j =n$ for $j \neq i$. We have then:
\begin{itemize}
\item[(5)] $K_Y + \mathcal{L}_{\psi} \equiv -L + E_i$ for some $i \in \{1,2,3,4\}$.
\end{itemize}

If $S = 2n$, then $\lambda_1 +\lambda_2 +\lambda_3 +\lambda_4 = 2n$, and we have then:
\begin{itemize}
\item[(6)] $K_Y + \mathcal{L}_{\psi} \equiv -L + E_i + E_j$ for some $i \neq j \in \{1,2,3,4\}$.
\end{itemize}

Finally, if $S = 3n$, then $\lambda_1 +\lambda_2 +\lambda_3 +\lambda_4 = n$, and we have:
\begin{itemize}
\item[(7)] $K_Y + \mathcal{L}_{\psi} \equiv -L + E_i + E_j+E_k$ where $i \neq j \neq k \in \{1,2,3,4\}$.
\end{itemize}

\subsection{{\color{darkgreen}$F(a_1, \ldots , a_5) = 3n$}}

Assume that $S=0$: then, by remark \ref{s=0}, we have:
$$
3n = F = 2([a_1] + [a_2]+[a_3]) \ \implies [a_1] + [a_2]+[a_3] = \frac{3n}{2}.
$$

If $n$ is odd, this is a contradiction, whence we can assume that $n$ is even. On the other hand, we have that $\lambda_i = n$ for all $i$, contradicting the fact that $\lambda_1 +\lambda_2 +\lambda_3 +\lambda_4 = 6n$. Hence this case is not possible.

If we assume that there is an $i \in \{1,2,3,4\}$ such that $\lambda_i = 2n$, then for the strict transforms of the three lines, say $L_1, L_2, L_3$, passing through $p_i$, we have $\psi(L_1)+ \psi(L_2) + \psi(L_3) \geq 2n$, hence for the other three lines we have $\psi(L_4)+ \psi(L_5) + \psi(L_6) \leq n$ (since $F=3n$). Therefore $\lambda_j \leq n$ for $j \neq i$.

\underline{$S=n$}: then $\lambda_1 +\lambda_2 +\lambda_3 +\lambda_4  = 5n$, hence there is an $i \in \{1,2,3,4\}$ such that $\lambda_i = 2n$ and $\lambda_j = n$ for $j \neq i$.
Therefore we have the possibility:
\begin{itemize}
\item[(8)] $K_Y + \mathcal{L}_{\psi} \equiv -E_i$ for $i  \in \{1,2,3,4\}$.
\end{itemize}

\underline{$S=2n$}: then $\lambda_1 +\lambda_2 +\lambda_3 +\lambda_4  = 4n$, hence either there are  $i , j \in \{1,2,3,4\}$ such that $\lambda_i = 2n$, $\lambda_j = 0$ and $\lambda_k=n$ for $k \neq i,j$, or $\lambda_i =n$ for all $i  \in \{1,2,3,4\}$.
Therefore we have the possibilities:
\begin{itemize}
\item[(9)] $K_Y + \mathcal{L}_{\psi} \equiv E_j - E_i$, for $i \neq j  \in \{1,2,3,4\}$;
\item [(10)] $K_Y + \mathcal{L}_{\psi} \equiv 0$.
\end{itemize}

\underline{$S=3n$}: then $\lambda_1 +\lambda_2 +\lambda_3 +\lambda_4  = 3n$, hence we have the possibilities:
\begin{itemize}
\item[(11)] $K_Y + \mathcal{L}_{\psi} \equiv E_j + E_i-E_k$, for $i \neq j \neq k  \in \{1,2,3,4\}$;
\item [(12)] $K_Y + \mathcal{L}_{\psi} \equiv E_i$ for $i  \in \{1,2,3,4\}$.
\end{itemize}

\subsection{{\color{darkgreen}$F(a_1, \ldots , a_5) = 4n$}}

If $S=0$, then by remark \ref{s=0}, we have:
$$
4n = F = 2([a_1] + [a_2]+[a_3]) \ \implies [a_1] + [a_2]+[a_3] = 2n.
$$
Moreover, $a_1 = a_4$, $a_2 = a_5$, hence $\lambda_i = 2n$ for all $i \in \{1,2,3,4\}$. This implies:
\begin{itemize}
\item[(13)] $K_Y + \mathcal{L}_{\psi} \equiv L  -E_1-E_2-E_3- E_4$.
\end{itemize}

Assume that there is an $i \in \{1,2,3,4\}$ such that $\lambda_i = 0$: then for the strict transforms of the three lines, say $L_1, L_2, L_3$, passing through $p_i$, we have $\psi(L_1)+ \psi(L_2) + \psi(L_3) < n$, hence for the other three lines we have $3n-3 \geq \psi(L_4)+ \psi(L_5) + \psi(L_6) > 3n$ (since $F=4n$), a contradiction. Therefore $\lambda_j \geq n$ for $i \in \{1,2,3,4\}$.

\underline{$S=n$}: then $\lambda_1 +\lambda_2 +\lambda_3 +\lambda_4  = 7n$, hence  we have the possibility:
\begin{itemize}
\item[(14)] $K_Y + \mathcal{L}_{\psi} \equiv L  -E_i-E_j-E_k$ for $i\neq j \neq k  \in \{1,2,3,4\}$.
\end{itemize}

\underline{$S=2n$}: then $\lambda_1 +\lambda_2 +\lambda_3 +\lambda_4  = 6n$, hence  we have:
\begin{itemize}
\item[(15)] $K_Y + \mathcal{L}_{\psi} \equiv L -E_i - E_j$, for $i \neq j  \in \{1,2,3,4\}$.
\end{itemize}

\underline{$S=3n$}: then $\lambda_1 +\lambda_2 +\lambda_3 +\lambda_4  = 5n$, hence we have:
\begin{itemize}
\item [(16)] $K_Y + \mathcal{L}_{\psi} \equiv L-E_i$ for $i  \in \{1,2,3,4\}$.
\end{itemize}

\subsection{{\color{darkgreen}$F(a_1, \ldots , a_5) = 5n$}}
Assume that there is an $i \in \{1,2,3,4\}$ such that $\lambda_i < 2n$: then for the strict transforms of the three lines, say $L_1, L_2, L_3$, passing through $p_i$, we have $\psi(L_1)+ \psi(L_2) + \psi(L_3) < 2n$, hence for the other three lines we have $3n-3 \geq \psi(L_4)+\psi(L_5) + \psi(L_6) > 3n$ (since $F=5n$), a contradiction. Therefore $\lambda_j = 2n$ for $i \in \{1,2,3,4\}$. Therefore we have:

\begin{itemize}
\item [(17)] $K_Y + \mathcal{L}_{\psi} \equiv 2L- E_1-E_2-E_3-E_4$ for $i  \in \{1,2,3,4\}$.
\end{itemize}

We shall use the following

\begin{proposition}\label{rank}
Let $n = 5$, or $n > 6$. Then for  $\psi \in ((\ZZ / n \ZZ)^5)^*$ we have $\rk\{D_{\sigma} : \psi(\sigma) \neq -1\} < 5$ if and only if 
$\{D_{\sigma} : \psi(\sigma) \neq -1\} = \{A, B_1, B_2, B_3\}$, where $B_1,B_2, B_3$ are pairwise disjoint and $AB_i=1$. In this case $K_Y + \mathcal{L}_{\psi} \equiv A$.
\end{proposition}

\begin{proof}
The first part is Proposition \ref{rankcond}. It remains to show that if $\{D_{\sigma} : \psi(\sigma) \neq -1\} = \{A, B_1, B_2, B_3\}$  then $K_Y + \mathcal{L}_{\psi} \equiv A$.

Using the $\mathfrak S_5$ symmetry, we may assume that $A$ is the strict transform of one of the six lines  in $\PP^2$,  $A= E_{34}$. 

Then we have:
$$
\begin{array}{lll}
 \psi(E_{14}) = a_1 = n-1 , & \psi(E_{24}) = a_2 = n-1  &  \psi(E_{23}) = a_4 = n-1 ,\\
 \psi(E_{13}) = a_5 = n-1 , & \psi(E_{45}) = [2-a_3] = n-1  &  \psi(E_{35}) = [2-a_3] = n-1, \\
\end{array}
$$
whence $a_3 =3$.
This implies that $F(n-1,n-1, 3,n-1,n-1) = 4n$ and $\lambda_1=\lambda_2 =2n$, $\lambda_3=\lambda_4 =n$. Hence $K_Y + \mathcal{L}_{\psi} \equiv E_{34}$.

\end{proof}

\begin{remark}\label{rank2}
1) For later use, we work out also the calculation in the case where $A=E_{i5}$, wlog $i=4$. 

Here we have: $\psi(E_{13}) = \psi(E_{23})=\psi(E_{12})=n-1$, i.e. $a_1=a_2=a_3 = n-1$. Moreover,

$$
\begin{array}{lll}
 \psi(E_{15}) = [a_2+a_3+a_4] =[a_4-2] =n-1 & \implies & a_4 =1,\\
  \psi(E_{25}) = [a_1+a_3+a_5] =[a_5-2] =n-1 & \implies & a_5 =1,\\
   \psi(E_{35}) = [-(a_3+a_4+a_5)] =n-1 .
\end{array}
$$

Then $F(n-1,n-1,n-1,1,1) = 3n$, $\lambda_1=\lambda_2 =\lambda_3 =n$ and $\lambda_4 = 0$. Therefore $K_Y + \mathcal{L}_{\psi} \equiv E_{45}$.

 Observe that $\rk\{D_{\sigma} : \psi(\sigma) \neq -1\} < 5$ implies that $F(a_1, \ldots, a_5) \in \{3n, 4n\}$.

2) If $n=6$ and $\{D_{\sigma} : \psi(\sigma) =  -1\}=\{E_{ij} : j \in \{1,2,3,4,5\} \setminus \{i\}\}$, for some $i$, whence 
we get $K_Y + \mathcal{L}_{\psi} \equiv X_i \equiv E_{jk}+E_{lm}$ for $\{j,k,l,m\} = \{1,2,3,4,5\} \setminus \{i\}$. In fact, by symmetry, we can wlog assume that 
$$
\{D_{\sigma} : \psi(\sigma) \neq  -1\} = \{E_{15}, E_{25}, E_{35}, E_{45}\}, \ K_Y+ \mathcal{L}_{\psi} \equiv X_5 .
$$
In this case $a_i = 5 \ \forall i$ and the calculation is straightforward.

3) If $n=4$, then $\rk\{D_{\sigma} : \psi(\sigma) \neq -1\} <5$ if and only if one of the following  conditions is satisfied:
\begin{itemize}
\item[a)] $\{D_{\sigma} : \psi(\sigma) \neq -1\}=\{E_{T_1}, \ldots , E_{T_5}\}$, where  there is a $j$ such that $T_i \in \{1, \ldots, 5\} \setminus \{j\}$;
\item[b)] $\{D_{\sigma} : \psi(\sigma) \neq -1\}=\{E_{T_1},E_{T_2}  , E_{T_3}\}$, where there are $i,j$ such that $T_1, T_2, T_3 \subset \{1, \ldots, 5\} \setminus \{i,j\}$.
\end{itemize}
In case b) we get  $K+ \mathcal{L}_{\psi} \equiv E_{ij}$.  Assume in fact $i=1, j=2$; then  $a_h = 3 \ \forall h \neq 3, a_3 =1,$ hence one easily sees that $K+ \mathcal{L}_{\psi} \equiv E_{12}$. 

In case a) instead, we have $K+ \mathcal{L}_{\psi} \equiv A-B$, where $B \in \{E_{T_1}, \ldots , E_{T_5}\}$,  $A \in \{D_{\sigma} : \psi(\sigma) = -1\}$, $AB=0$ and $A(E_{T_1} + \ldots +E_{T_5}) = 2$.

In fact, by symmetry, we can wlog assume that in case a)
$$
\{D_{\sigma} : \psi(\sigma) \neq -1\} = \{E_{14}, E_{12},E_{15},E_{25},E_{45} \};
$$
 then  $a_h = 3 \ \forall h = 2,3,4,5, a_1 \neq 1,3 $. For $a_1=0$ we get   $K_Y+ \mathcal{L}_{\psi} \equiv E_{35} -E_{15}$, 
$a_1=2$ we get   $K_Y+ \mathcal{L}_{\psi} \equiv E_{34} -E_{45}$.
\end{remark}

\section{Cohomology of logarithmic differential forms}

In this section we consider the following situation: let $P_1, \ldots , P_m$ be $m$ distinct points in $\PP^2$ and let
\begin{itemize}
\item $\pi \colon Y�:= \hat{\PP}^2(P_1, \ldots , P_m) \rightarrow \PP^2$ be the blow-up of the projective plane in $P_1, \ldots , P_m$;
\item $D_1, \ldots , D_N \subset \PP^2$ be smooth rational curves, such that
\item $D_1 + \ldots + D_N$ has global normal crossings.
\end{itemize}

The aim of this paragraph is to prove the following general vanishing theorem:

\begin{theorem}\label{gvt}
Let $\mu, l,r,k \in \NN$ such that 
$$
0 \leq \mu \leq l \leq r \leq k \leq N,
$$
and consider
\begin{itemize}
\item $A: = D_{l+1} + \ldots +D_r + \ldots +D_k$,
\item $B: = D_1 + \ldots +D_{\mu}$,
\item $D:=A-B$,
\item $\mathcal{F}:=\Omega_Y^1(\log D_1, \ldots , \log D_r)(D)$.
\end{itemize}

Assume that
\begin{enumerate}
\item $H^2(Y,\mathcal{F}) = 0$;
\item for $l+1 \leq i \leq r$ we have  $D_i(A-B) \geq -1$;
\item for $l+1 \leq i \leq k$ we have 
$$
D_i(\sum_{\nu = 1}^k D_{\nu} -B) \geq 1 ;
$$
\item 
$$
\rk \{ D_i | 1 \leq i \leq k,  D_i D_1 = \ldots = D_iD_{\mu}=0 \} \geq m+1-\mu +R,
$$
where
$$
R:= \sum_{i = \mu+1,  \exists j \in \{1, \ldots, \mu\} : D_j D_i \neq 0}^k(D_i(D_1 +\ldots +D_{\mu}) -1).
$$
\end{enumerate}
Then $H^1(Y, \mathcal{F}) = 0$.
\end{theorem}

\begin{proof}
Consider
$$
\mathcal{G}:=\Omega_Y^1(\log D_1, \ldots , \log D_l)(D).
$$
Then we have the exact residue sequence 
$$
0 \rightarrow \mathcal{G} \rightarrow \mathcal{F} \rightarrow \bigoplus_{i=l+1}^r \hol_{D_i}(A-B) \rightarrow 0.
$$
By (2), for all $l+1 \leq i \leq r$,  we have that  $\hol_{D_i}(A-B) \cong \hol_{\PP^1}(n)$, where $n \geq -1$. Therefore by the long exact cohomology sequence  it follows:
\begin{itemize}
\item $H^2(Y, \mathcal{G}) = 0$ (since $H^2(Y, \mathcal{F}) = 0$),
\item if $H^1(Y, \mathcal{G}) = 0$, then $H^1(Y, \mathcal{F}) = 0$.
\end{itemize}
Hence it suffices to show that $H^1(Y, \mathcal{G}) = 0$.

Set now
$$
\mathcal{G'}:=\Omega_Y^1(\log D_1, \ldots , \log D_k)(-B)
$$
and consider the exact sequence (cf. e.g. \cite{ev}, p. 13)
$$
0 \rightarrow \mathcal{G'} \rightarrow \mathcal{G} \rightarrow \bigoplus_{i=l+1}^k \Omega^1_{D_i}(D_1 + \ldots +D_k-B) \rightarrow 0.
$$

Since $D_i \cong \PP^1$, we have 
$$
\Omega^1_{D_i}(D_1 + \ldots +D_k-B) \cong \hol_{\PP^1}(D_i(K_Y+D_i +D_1 + \ldots +D_k-B)) \cong \hol_{\PP^1}(n),
$$
where by (3), $n= -2 + D_i(D_1 + \ldots +D_k-B) \geq -1$.

Therefore it follows:
\begin{itemize}
\item $H^2(Y, \mathcal{G'}) = 0$ (since $H^2(Y, \mathcal{G}) = 0$),
\item if $H^1(Y, \mathcal{G'}) = 0$, then $H^1(Y, \mathcal{G}) = 0$,
\end{itemize}
and it suffices to show that $H^1(Y, \mathcal{G'}) = 0$.

Using the analog of lemma \ref{chi} for $m$ blow ups we calculate

 \begin{multline}
\chi(\mathcal{G'}) = B^2 - m - 1 + \sum_{i=1}^k(1-B D_i)= \\
= (D_1 + \ldots +D_{\mu})^2 - m - 1 
+ \sum_{i=1}^k(1-D_i(D_1 + \ldots +D_{\mu})) = \\
= -m-1+ \mu + \sum_{i=\mu + 1}^k(1-D_i(D_1 + \ldots +D_{\mu})) 
= (\sum_{i = \mu + 1, D_1D_i = \ldots = D_{\mu}D_i=0}^k 1) -m-1+ \mu -R.
\end{multline}

On the other hand, consider the commutative diagram:

\medskip

 \begin{equation}\label{diag1}
\xymatrix{
0 \ar[d]&  0 \ar[d]  \\
H^0(\mathcal{G'}) \ar[d] \ar[r] &H^0(\Omega_Y^1(\log D_1, \ldots , \log D_k)) \ar[d] \\
\bigoplus_{i : D_iD_1 =\ldots = D_iD_{\mu}=0} H^0(\mathcal{O}_{D_i})\ar[r]&\bigoplus_{i=1}^k H^0(\mathcal{O}_{D_i})\ar[d] \\
& H^1(\Omega_Y^1)
}
\end{equation}

Since the homomorphism in $ H^1(\Omega_Y^1) $ takes the constant function equal to $1$ on $D_i$ to the Chern class of $D_i$, this  implies that 
\begin{multline}
h^0(Y, \mathcal{G'}) \leq (\sum_{i =1, D_1D_i = \ldots = D_{\mu}D_i=0}^k 1) - \rk \{ D_i : D_i D_1 = \ldots D_iD_{\mu}=0, 1 \leq i \leq k\} \leq \\
\leq ( \sum_{i =1, D_1D_i = \ldots = D_{\mu}D_i=0}^k 1) -m-1+ \mu -R = \chi(\mathcal{G'}),
\end{multline}
where the last inequality holds by assumption (4).

Since $h^2(Y, \mathcal{G'})  = 0$ the inequality  $h^0(Y, \mathcal{G'}) \leq \chi(\mathcal{G'}) = h^0(Y, \mathcal{G'}) - h^1(Y, \mathcal{G'}) $,  implies that $h^1(Y, \mathcal{G'}) =0)$ and the theorem is proven.

\end{proof}

\begin{remark}
Note that the assumption "(1) $h^2(\mathcal{F})=0$" in the previous Theorem is always satisfied in our applications by Remark \ref{h2=0}.
\end{remark}
 
\section{Proof of Theorem \ref{main} for $n \neq 4,6$}

In this section we use the classification of all the possibilities for $K+\mathcal{L}_{\psi}$ in order to prove our main result.
Let $\psi \in ((\ZZ/ n \ZZ)^5)^*$.

\subsection{{\color{darkgreen}$F(a_1, \ldots , a_5) = n$}}
Observe that by remark \ref{rank2} we know that $\rk\{D_{\sigma} : \psi(\sigma) \neq -1\} =5$.
We have two cases here:
\begin{itemize}
\item[(1)] $K_Y + \mathcal{L}_{\psi} \equiv -2L + E_i + E_j + E_k$, where $i,j,k \in \{1,2,3,4\}$ pairwise different;
\item[(2)] $K_Y + \mathcal{L}_{\psi} \equiv -2L  + E_1 + E_2+E_3+ E_4$.
\end{itemize}
In case (1) we can assume wlog that 
$$
K_Y + \mathcal{L}_{\psi} \equiv -2L + E_1 + E_2 + E_3 \equiv -X_5 -E_4 \equiv -E_{13} -E_{24}-E_{45}.
$$
Hence  $ \mathcal{L}_{\psi}  = L - E_4$,  $\lambda_4 = n$ and $F(a_1, \ldots , a_5) = n$, therefore  we have $\psi(E_{i4}) =0 \ \forall i \leq 3$, i.e. $a_1=a_2=a_3 = 0$, 
which implies also $\psi(E_{45}) = 0$. Since $[a_4]+[a_5]+[-(a_1+ \ldots + a_5)]  = n$, we have  $\psi(L) =  -1$ for at most one of  the strict transforms of lines in $\PP^2$ contained in $D$; in case there is one such line, assume  wlog $\psi(E_{23}) = a_4 = n-1$ and for all the others $\psi \neq -1$. This implies that we have
$$
\{D_{\sigma} : \psi(\sigma) =  -1\} = \emptyset \ \ \rm{or} \ \  \{D_{\sigma} : \psi(\sigma) =  -1\} = \{E_{23}, E_{15}\}.
$$

We dispose of case (1) via the following

\begin{proposition}
\begin{multline}
h^1(Y, \Omega_Y^1(\log E_{ij} : \{i,j\} \subset \{1,2,3,4,5\})(-E_{13} -E_{24}-E_{45})) = \\
= h^1(Y, \Omega_Y^1(\log E_{ij} :  \{i,j\} \neq \{2,3\}, \{1,5\})(-E_{13} -E_{24}-E_{45})) = 0.
\end{multline}
\end{proposition}
\begin{proof}

We have to verify in both cases the assumptions of theorem \ref{gvt}. Note that (1) is automatically satisfied by remark \ref{h2=0} and (2), (3) are empty.  Observe that the sets
\begin{multline}
\{E_{ij} : \{i,j\} \subset \{1,2,3,4,5\}, E_{ij}E_{13} =E_{ij}E_{24}=E_{ij}E_{45}= 0\} = \\
= \{E_{ij} : \{i,j\}\neq \{2,3\}, \{1,5\}, E_{ij}E_{13} =E_{ij}E_{24}=E_{ij}E_{45}= 0\} = \{E_{14}, E_{34}\}.
\end{multline}

have rank 2, hence (4) is satisfied, since, for all $E_{ij}$ not a component of $B$,   $E_{ij} B \leq 1$ ($B$ consists of a reducible fibre plus a section of the pencil 
$X_5$).

This proves the claim.

\end{proof}

In case (2) we have $K_Y + \mathcal{L}_{\psi} \equiv -X_5$. Since $\lambda_i = 0$ for all $i \in \{1,2,3,4\}$, if for the strict transform of a line $L$ we have $\psi(L) = -1$, wlog $\psi(E_{14}) = a_1 = n -1$, we infer  then that $a_2, a_3, a_5, [-(a_1 + \ldots +a_5)] = 0$, $a_4 = 1$. 
In this case we have then
$$
 \{D_{\sigma} : \psi(\sigma) =  -1\} = \{E_{14}, E_{25}, E_{35}\}.
$$
 
If instead for all lines $L$ we have $\psi(L) \neq  -1$ wlog we may assume (since $S(a_1, \cdots, a_5) = 2n$)  that 
$$
\{D_{\sigma} : \psi(\sigma) =  -1\} \subset \{ E_{25}, E_{35}\}.
$$

Since $- X_5 \equiv - B : = - (E_{13} +E_{24})$, in both cases  assumptions (1)-(3) of theorem \ref{gvt} are satisfied.  There remains to show (4). It is
first of all  immediate
to see that $R=0$ ($B$ is a reducible fibre of $X_5$).

Observe that
$$
\{E_{ij} : \{i,j\} \subset \{1,2,3,4,5\}, E_{ij}E_{13} =E_{ij}E_{24}= 0\} =  \{E_{14}, E_{34}, E_{12}, E_{23}\}.
$$
This set has rank 3, and  if we remove from this set its intersection with $\{ E_{25}, E_{35}\}$, respectively $\{E_{14}, E_{25}, E_{35}\}$,
we obtain either the full set or  the set
$ \{E_{34}, E_{12}, E_{23}\}$, which 
has rank 3. Whence (4) is satisfied in both cases.

\subsection{{\color{darkgreen}$F(a_1, \ldots , a_5) = 2n$}}
 Consider the first subcase which, up to symmetry of the four blown up points, can be written as 

\begin{itemize}
\item[(3)] $K_Y + \mathcal{L}_{\psi} \equiv -L - E_4 + E_1 + E_2 +E_3  \equiv -E_{j4} -E_{45} + E_{j5},$
\end{itemize}
for each  $ j \in  \{1,2,3\}$. 

 Since $\lambda_4=2n$, $\la_1 = \la_2 = \la_3 = 0$,  we have 
$\psi(E_{23}) + \psi(E_{13}) + \psi(E_{12}) = 2n$ and $\psi(E_{14}) = \psi(E_{24})=\psi(E_{34}) = 0$ (since $F = 2n$). I.e., $a_1=a_2=a_3=0$, $\psi(E_{45}) = 0$. Moreover, $[a_4]+[a_5]+[-(a_1+ \ldots + a_5)]  = 2n$. Therefore we have the following possibilities:
\begin{itemize}
\item[i)] $[a_4], [a_5], [-(a_1+ \ldots + a_5)] = [-(a_4+a_5)] \neq -1$, hence $\psi(D_{\sigma}) \neq -1$ for all $\sigma$;
\item[ii)] wlog  $[a_5],  [-(a_4+a_5)] \neq -1$, $[a_4] =-1$; hence $\{D_{\sigma} : \psi(\sigma) =  -1\} = \{E_{23}, E_{15}\}$;
\item[iii)] wlog $[-(a_4+a_5)] \neq -1$, $[a_4] = [a_5]=-1$; hence $\{D_{\sigma} : \psi(\sigma) =  -1\} = \{E_{23}, E_{13},E_{15}, E_{25}\}$.
\end{itemize}

Observe that  cases ii), iii) occur in particular for $n=5$:
\begin{itemize}
\item[ii)] $\psi=(0,0,0,4,3,3)$,
\item[iii)]  $\psi=(0,0,0,4,4,2)$.
\end{itemize}

Since for $n=5$ the Hirzebruch-Kummer covering is a ball quotient, and ball quotients are infinitesimally rigid by \cite{CalabiV}, p. 500.
\begin{multline}
h^1(Y, \Omega_Y^1(\log D_{\sigma} : D_{\sigma} \neq E_{23}, E_{13},E_{15}, E_{25})(-E_{j4} -E_{45} + E_{j5})) = \\
= h^1(Y, \Omega_Y^1(\log D_{\sigma} : D_{\sigma} \neq E_{23}, E_{15})(-E_{j4} -E_{45} + E_{j5})) = 0.
\end{multline}

For i) consider instead the exact sequence
\begin{multline}
0 \rightarrow \mathcal{F}' :=  \Omega_Y^1(\log D_{\sigma} : D_{\sigma} \neq E_{23}, E_{15})(-E_{j4} -E_{45} + E_{j5}) \rightarrow \mathcal{F} :=\\
= \Omega_Y^1(\log D_{\sigma} : D_{\sigma} \neq 0)(-E_{j4} -E_{45} + E_{j5})\rightarrow \\
\rightarrow\mathcal{O}_{E_{23}}(-E_{j4} -E_{45} + E_{j5}) \oplus  \mathcal{O}_{E_{15}}(-E_{j4} -E_{45} + E_{j5}) \cong   \mathcal{O}_{\PP^1}(-1) \oplus \mathcal{O}_{\PP^1}(-1)\rightarrow 0.
\end{multline}

Since $h^1(\mathcal{F}') = 0$ (by ii)), it follows that $h^1(\mathcal{F}) = 0$, and we have proven case i). 

The next case is:
\begin{itemize}
\item[(4)] $K_Y + \mathcal{L}_{\psi} \equiv -L \equiv -X_i -E_{i5}$, $ \forall i \in\{1,2,3,4\}$.
\end{itemize}
Here $\la_i = n \ \forall i \Rightarrow S (a_1, \dots, a_5) = 0$ hence by Remark \ref{s=0}, we have $a_1 = a_4$, $a_2 = a_5$, $a_3 = [-(a_1+a_2)]$ and $[a_1] + [a_2]+[a_3] = n$. This implies that $\psi(E_{i5})= 0$ for all $i$ and we have two cases:
\begin{itemize}
\item[i)] $[a_1], [a_2], [a_3]  \neq -1$, hence $\psi(D_{\sigma}) \neq -1$ for all $\sigma$;
\item[ii)] wlog $[a_1] =-1$; $[a_2],  [a_3] \neq -1$, hence $\{D_{\sigma} : \psi(\sigma) =  -1\} = \{E_{23},E_{14}\}$.
\end{itemize}

These cases  occur in particular  for $n=5$:
\begin{itemize}
\item[i)] $\psi=(1,2,3,1,2)$,
\item[ii)]  $\psi=(4,1,0,4,1)$.
\end{itemize}

Therefore, by the same argument as above, we get
\begin{multline}
h^1(Y, \Omega_Y^1(\log D_{\sigma} : D_{\sigma} \neq 0)(-X_i -E_{i5})) = \\
= h^1(Y, \Omega_Y^1(\log D_{\sigma} : D_{\sigma} \neq E_{23},E_{14})(-X_i -E_{i5})) = 0.
\end{multline}

Consider the case
\begin{itemize}
\item[(5)] $K_Y + \mathcal{L}_{\psi} \equiv - X_i$ for some $i \in \{1,2,3,4\}$.
\end{itemize}
Wlog we can assume $i=1$, i.e., $\lambda_1 =0$ and $\lambda_j=n$ for $j \neq 1$. Then 
$$
K_Y + \mathcal{L}_{\psi} \equiv -X_1 \equiv -E_{23}-E_{45}  \equiv -E_{24}-E_{35} \equiv -E_{34}-E_{25}.
$$
According to theorem \ref{gvt} it suffices  to show (since then $A=0, R=0$):
\begin{itemize}
\item there is a decomposition $X_1 = D_1 + D_2$, such that $\psi(D_1), \psi(D_2) \neq n-1$,
\item $M:= \{D_{\sigma} : \psi(\sigma) \neq n-1, D_{\sigma}D_1 = D_{\sigma}D_2 = 0 \}$ has rank 3.
\end{itemize}

Since $\lambda_1=0$, we have (assume as always $ 0 \leq a_i < n$) $a_2+a_3+a_4 <n$. If one of the $a_2,a_3,a_4$ equals $n-1$, say $a_4 = n-1$, then $a_2=a_3=0$. In 
any case, $\psi(E_{34}), \psi(E_{24}) \neq n-1$.

Since $S= n$ there exists at most one $i \in \{1,2,3,4\}$ such that $\psi(E_{i5}) = n-1$. This implies that either $\psi(E_{25}) \neq n-1$ or $\psi(E_{15}) \neq n-1$, hence 
$(D_1,D_2) = (E_{34},E_{25})$ or $(D_1,D_2) = (E_{24},E_{15})$ satisfy the first condition above. We can assume wlog that $\psi(E_{34}), \psi(E_{25}) \neq n-1$.

If $\psi(E_{24}), \psi(E_{35}), \psi(E_{23}), \psi(E_{45})  \neq n-1$, we are done since $\{E_{24}, E_{35}, E_{23}, E_{45} \} \subset M$, which implies that $M$ has rank 3.

We  show now  that it suffices to prove that   the case $\psi(E_{23}) = a_4 = n-1 = \psi(E_{45}) = n-1$ cannot occur. Since,  if $\psi(E_{23}) \neq  n-1$ and $\psi(E_{45}) = n-1$, then the set $\{E_{24}, E_{35}, E_{23} \} \subset M$, which has rank three.  If instead   $\psi(E_{23}) = a_4 = n-1$ and $\psi(E_{45}) \neq n-1$, we have already remarked 
that $a_2=a_3 =0$. Since $\lambda_3 = n$,  $\psi(E_{35}) = n-1$ if and only if $a_1  +a_6 = 2n-1$, which is not possible. Hence $\{E_{24}, E_{35},  E_{45} \} \subset M$.

Let's show that   $\psi(E_{23}) = a_4 = n-1=\psi(E_{45}) = n-1$ cannot occur. Otherwise we have $a_4 + a_5 + a_6 = \la_4 + [\psi (E_{45}) ]  = 2n-1$, whence $a_5+a_6 = n$. $F=2n$ implies $a_1 =1$. On the other hand, since $\lambda_2 = \lambda_3 =n$ we have  $a_1 + a_5 \geq n$ and $a_1 + a_6 \geq n$, a contradiction.

The next case is:
\begin{itemize}
\item[(6)] $K_Y + \mathcal{L}_{\psi} \equiv -L + E_i + E_j$ for $i \neq j \in \{1,2,3,4\}$.
\end{itemize}
Wlog $i=1, j=2$, i.e., $K_Y + \mathcal{L}_{\psi} \equiv -L + E_1 + E_2  \equiv -E_{34}$.

Assume that  $\psi(E_{34}) = a_3 = n-1$: then (since $\lambda _1 = \lambda_2=0$) $a_2=a_4=a_1=a_5 = 0$, which contradicts $F(a_1, \dots, a_5) =2n$. Therefore $\psi(E_{34}) \neq n-1$ and by theorem \ref{gvt} it suffices  to show: $M:=\{D_{\sigma} : \psi(\sigma) \neq n-1, D_{\sigma}E_{34} = 0 \}$ has rank at least 4.

Note that:
\begin{itemize}
\item $F=2n$ $\implies$ at most two of the $E_{ij}$'s, $\{i,j\} \subset \{1,2,3,4\}$ have $\psi(E_{ij}) = n-1$;
\item $S=2n$ $\implies$ at most two of the $E_{i5}$'s, $i \in \{1,2,3,4\}$ have $\psi(E_{i5}) = n-1$.
\end{itemize}

Assume that there is one of the strict transforms of the lines orthogonal to $E_{34}$ which has $\psi=n-1$,  wlog  $\psi(E_{14}) = a_1 = n-1$. Then $\lambda_2 = 0$ implies $a_3=a_5 = 0$ and $\psi(E_{25}) = n-1$. In particular, $\psi(E_{13}) =a_5 \neq n-1$.
Moreover, $n+\psi(E_{45})=\la_4 + \psi(E_{45}) = a_4+a_5+a_6 =a_4+a_6 \leq 2n-2$, hence $\psi(E_{45}) \neq n-1$.

Hence $\{E_{13}, E_{45}\} \subset M$.

Assume that also $\psi(E_{24}) = a_2 =n-1$. Then $a_4=a_3=0$, since $\la_1 =0$. But then, since $\la_4=n$,  $n \leq  a_4+a_5+a_6 =a_6$, a contradiction. Hence $\psi(E_{24}) \neq n-1$.

Assume that  $\psi(E_{23}) = a_4 =n-1$. Then $a_2=a_3=0$ and $a_6=2$ and this implies that $\psi (E_{35}) \neq n-1$. Therefore
$$
\{E_{13}, E_{45}, E_{35}, E_{24}\} \subset M,
$$
hence $\rk M \geq 4$.

If instead $\psi(E_{23}) \neq n-1$ (and $\psi(E_{14})  = n-1$), then
$$
\{E_{13}, E_{45}, E_{23}, E_{24}\} \subset M,
$$
and again $\rk M \geq 4$.

We can therefore assume that for all the strict transforms of the lines orthogonal to $E_{34}$: $\psi \neq n-1$, i.e.

$$
\{E_{14}, E_{23}, E_{13}, E_{24}\} \subset M.
$$
We need to show that either $\psi(E_{35}) \neq n-1$ or $\psi(E_{45}) \neq n-1$, then we are done.

Assume that $\psi(E_{35})=\psi(E_{45})=  n-1$. Since $S=2n$, this implies that $\psi(E_{15})+\psi(E_{25}) =2$. However
 $\psi(E_{15})+\psi(E_{25}) = a_2 + a_3 + a_4 + a_1 + a_3 + a_5 = F(a_1, \dots, a_5) + a_3 - a_5 \geq n+1$, a contradiction.

The last case in this subsection is
\begin{itemize}
\item[(7)] $K_Y + \mathcal{L}_{\psi} \equiv -L + E_i + E_j+E_k$ for $i \neq j \neq k \in \{1,2,3,4\}$.
\end{itemize}

We have $F = 2n$, $S=3n$ and 
wlog $\lambda_4 = n$, $\lambda_1 = \lambda_2 = \lambda_3 =0$. Then 
$$
K_Y + \mathcal{L}_{\psi} \equiv -L + E_1 + E_2+E_3 \equiv  -E_{i4} +E_{i5}
$$ 
for all $i \in \{1,2,3\}$. 

We set $A:= E_{i5}$, $B:= E_{i4}$ and we want to apply theorem \ref{gvt}, whence we have to verify that there is an $i$ such that $\psi(E_{i4}) \neq n-1$ and that 
assumptions (2)-(4) of theorem \ref{gvt} are satisfied.

We have for all $i  \in \{1,2,3\}$ that $E_{i5}(E_{i5}-E_{i4}) = -1$, i.e., (2) holds. Since $\lambda_1 = \lambda_2=\lambda_3 = 0$, it follows that $\psi(E_{j4}) = a_j \neq n-1$ for all $j  \in \{1,2,3\}$, otherwise if e.g. $a_1=n-1$, then $a_3=a_5=a_2=a_6=0$, contradicting $F(a_1, \dots, a_5)=2n$. 

 (3) is now satisfied since $E_{j4}E_{i5} =1$ for $i \neq j$, hence
$$
E_{i5}((\sum_{\sigma : \psi(\sigma) \neq n-1}^{D_{\sigma} \neq E_{i5}} D_{\sigma})  +E_{i5}-E_{i4}) = (\sum_{\sigma : \psi(\sigma) \neq n-1}^{D_{\sigma} \neq E_{i5}}D_{\sigma})E_{i5} -1 \geq 2-1 = 1.
$$

It remains to show (4), i.e. $M_i:=\{D_{\sigma} : \psi(\sigma) \neq n-1, D_{\sigma}E_{i4} = 0 \} \cup \{E_{i5} \}$ has rank 4 for some $i  \in \{1,2,3\}$.

Observe that since $F=2n$, at most two of the $E_{ij}$, $\{i,j\} \subset \{1,2,3\}$ have $\psi(E_{ij}) = n-1$. Assume wlog $\psi(E_{23}) = a_4 = \psi(E_{13}) = a_5 =n-1$. Since $\lambda_1 =\lambda_2 = 0$, this implies that $a_1 = a_3 = a_2 = 0$, whence $a_6 = 2$. But this contradicts $\lambda_4 = n$.

Assume now that for exactly one  of the $E_{ij}$, $\{i,j\} \subset \{1,2,3\}$ have $\psi(E_{ij}) = n-1$,   wlog $\psi(E_{23}) =a_4 = n-1$. Then, since $\lambda_1 =0$, $a_2 = a_3 = 0$ and $\psi(E_{15}) = n-1$.

If $\psi(E_{i5}) \neq n-1$ for $i \neq 1$, then we are done since e.g.
$\{E_{14},  E_{34}, E_{25},  E_{45} \} \subset M_2$.

Suppose that also $\psi(E_{45}) = n-1$. Then (since $\lambda_4 = n$): 
$$
2n-1 = a_4+a_5+a_6 = n-1 + a_5+a_6 \ \implies \ a_5 + a_6 = n.
$$

$F=2n$ implies that $a_1 = 1$, in particular $\psi(E_{25})$ or  $\psi(E_{35}) \neq n-1$. If $\psi(E_{35})\neq n-1$,  write $K_Y + \mathcal{L} = E_{25}-E_{24}$
and observe that 
$$
\{E_{14},  E_{34}, E_{35},  E_{12} \} \subset M_2.
$$
Therefore we may assume that  $\psi(E_{ij}) \neq n-1$ for all $\{i,j\} \subset \{1,2,3\}$. 

If for some $i \leq 3$ we have $\psi(E_{i5}) \neq n-1$ we are done since, if $(i,j,k)$ is a permutation of $(1,2,3)$,
$$
\{E_{24},  E_{34}, E_{ij},  E_{ik}, E_{i5} \} \subset M_i.
$$

If  $\psi(E_{45}) \neq n-1 $, then 
$$
\{E_{24},  E_{34}, E_{13},  E_{12}, E_{45} \} \subset M_1.
$$

We conclude as follows: the case that $\psi(E_{i5}) =  n-1$ for $i=1,2,3,4$ contradicts  $S = 3n$.

\subsection{{\color{darkgreen}$F(a_1, \ldots , a_5) = 3n$}}
Here the first case is:
\begin{itemize}
\item[(8)] $K_Y + \mathcal{L}_{\psi} \equiv -E_{i5}$ for $i  \in \{1,2,3,4\}$.
\end{itemize}

We have $F=3n$, $S=n$. Wlog $i=1$, then $\lambda_1 = 2n, \lambda_2 = \lambda_3 = \lambda_4 =n$. In particular, $\psi(E_{15}) \neq n-1$ by Remark \ref{lambda2n}.
Since $A=0$, $B$ is irreducible, to apply  Theorem \ref{gvt} it suffices  to verify that $M:=\{D_{\sigma} : \psi(\sigma) \neq n-1, D_{\sigma}E_{15} = 0 \}$ has rank 4.

Since $S= n$ we have that $\psi(E_{i5}) = n-1$ for at most one $i$. Assume wlog that $\psi(E_{25}) = n-1$. Since $\lambda_2 =n$ this implies 
$$
2n-1 = a_1 + a_3 +a_5 \ \implies \ a_2 +a_4+ a_6 = n+1.
$$
Since $\la_1 = 2n$, it  follows that $a_2 + a_3 +a_4 \geq 2n$, hence $a_2 +a_4 \geq n+1 \Rightarrow a_2 +a_4 = n+1, a_6 = 0, a_3 = n-1$.

Hence at least one of $a_1 , a_5$ is different from  $n-1$. Then either $\{E_{35},  E_{45}, E_{12},  E_{13} \} \subset M$
or $\{E_{35},  E_{45}, E_{12},  E_{14} \} \subset M$ and we are done, unless 
 $\psi(E_{i5}) \neq n-1$ for all i. 
 
 In this case  $E_{35},  E_{45}, E_{25} \in M$ and it suffices to show that it cannot happen that $\psi(E_{14}) = \psi(E_{12}) =\psi(E_{13}) =n-1$, i.e. 
 $a_1=a_6=a_5 = n-1$. But then (since $F=3n$) $a_2+a_3+a_4 = 3$ contradicting $\lambda_1 = 2n$.

\begin{remark}\label{intersection}
If $\lambda_i = 0$, then for the strict transforms of the three lines, say $L_1, L_2, L_3$, passing through $p_i$, we have $\psi(L_1)+ \psi(L_2) + \psi(L_3) \leq n-1$, hence at least two of these lines have $\psi(\ldots) \neq n-1$. Therefore  $E_{i5}(\sum_{\psi(\sigma) \neq -1, D_{\sigma} \neq E_{i5}} D_{\sigma}) \geq 2$.
\end{remark}

The next cases are:
\begin{itemize}
\item[(9)] $K_Y + \mathcal{L}_{\psi} \equiv E_{j5} - E_{i5}$, for $i \neq j  \in \{1,2,3,4\}$;
\item [(10)] $K_Y + \mathcal{L}_{\psi} \equiv 0$.
\end{itemize}

For case (9) we can assume wlog $j=1, i=2$, i.e. $K_Y + \mathcal{L}_{\psi} \equiv E_{15} - E_{25}$.

Observe  again that $\psi(E_{25}) \neq n-1$, since $\lambda_2 = 2n$. Moreover, 
(2), (3) of Theorem \ref{gvt} are satisfied, since $E_{15}(E_{15} - E_{25}) = -1$ and 

$$
E_{15}((\sum_{\sigma : \psi(\sigma) \neq n-1}^{D_{\sigma} \neq E_{15}} D_{\sigma})  +E_{15} -E_{25}) = (\sum_{\sigma : \psi(\sigma) \neq n-1}^{D_{\sigma} \neq E_{15}}  D_{\sigma})E_{15} -1 \geq 2-1 = 1.
$$
The above inequality follows by remark \ref{intersection}.

We need to show: $M:=\{D_{\sigma} : \psi(\sigma) \neq n-1, D_{\sigma}E_{25} = 0 \} \cup \{E_{15} \}$ has rank 4.  Observe that
$$
M \subset \{D_{\sigma} : D_{\sigma}E_{25} = 0 \}  = \{  E_{24},  E_{23}, E_{12},  E_{15}, E_{35}, E_{45}  \}.
$$
Assume that $\psi(E_{24})=a_2 = n-1$. Then, since $\lambda_1=0$, we have $a_3=a_4=0$. But this contradicts $\lambda_2 =2n$. The same argument for $E_{23}$ shows that $\psi(E_{24}), \psi(E_{23}) \neq n-1$.

Assume that $\psi(E_{34}) = a_3 = n-1$. Then again $\lambda_1 = 0$ implies that $a_2 = a_4 = 0$. Since $\lambda_3 = \lambda_4 = n$ it follows that $\psi(E_{35}), \psi(E_{45}) \neq n-1$ and then 
$$
\{  E_{15},  E_{23}, E_{24},  E_{35}, E_{45}  \}\subset M,
$$
which implies that $\rk M = 4$.

Therefore we can assume that $\psi(E_{34}) \neq n-1$. Assume now that $\psi(E_{12}) = a_6 = n-1$. Then:
\begin{itemize}
\item $\psi(E_{35}) =n-1$ $\implies$ $a_1 + a_2 = n$;
\item $\psi(E_{45}) =n-1$ $\implies$ $a_4 + a_5 = n$.
\end{itemize}
If both equalities occur,  $F=3n \Rightarrow a_3 = 1$, which contradicts $\lambda_2 = 2n$. Therefore either $\psi(E_{45}) \neq n-1$ or $\psi(E_{35}) \neq n-1$.
We can assume wlog $\psi(E_{45}) \neq n-1$. Then:
$$
\{  E_{15},  E_{23}, E_{24},  E_{45}  \}\subset M,
$$
which implies that $\rk M = 4$.

Therefore we can assume that  $\psi(E_{12}) = a_6 \neq n-1$ and we are done since $\{  E_{15},  E_{23}, E_{24},  E_{12}  \}\subset M$.

For case (10) we are done by Theorem \ref{gvt}, since by  Proposition \ref{rank} it holds $\rk\{D_{\sigma} : \psi(\sigma) \neq -1\} = 5$.

In this subsection we are left with the following two cases:

\begin{itemize}
\item[(11)] $K_Y + \mathcal{L}_{\psi} \equiv E_{j5} + E_{i5}-E_{k5}$, for $i \neq j \neq k  \in \{1,2,3,4\}$;
\item [(12)] $K_Y + \mathcal{L}_{\psi} \equiv E_{i5}$ for $i  \in \{1,2,3,4\}$.
\end{itemize}

In case (11) we can assume wlog that $K_Y + \mathcal{L}_{\psi} \equiv E_{15} + E_{25}-E_{35}$. We have $F=S=3n$, $\lambda_4 =  n$, $\lambda_1 = \lambda_2 = 0$ and $\lambda_3 = 2n$.

Observe that again $\psi(E_{35}) \neq n-1$ since $\lambda_3 = 2n$. Moreover, 
(2), (3) of Theorem \ref{gvt} are satisfied, since $E_{i5}(E_{15}+E_{25} - E_{35}) = -1$  for $i=1,2$ and (by Remark \ref{intersection})

$$
E_{i5}((\sum_{\sigma : \psi(\sigma) \neq n-1}^{D_{\sigma} \neq E_{15}, E_{25}} D_{\sigma})  +E_{15}+E_{25} - E_{35}) = (\sum_{\sigma : \psi(\sigma) \neq n-1}^{D_{\sigma} \neq E_{15}, E_{25}}  D_{\sigma})E_{i5} -1 \geq 2-1 = 1.
$$
We need to show that $M:=\{D_{\sigma} : \psi(\sigma) \neq n-1, D_{\sigma}E_{35} = 0 \} \cup \{E_{15}, E_{25} \}$ has rank 4.  Observe that
$$
M \subset \{D_{\sigma} :D_{\sigma}E_{35} = 0  \} \cup \{E_{15}, E_{25}  \} = \{  E_{15},  E_{25}, E_{45},  E_{34}, E_{23}, E_{13}  \},
$$
and $\psi(E_{35}) \neq n-1$.

Since $\lambda_1 = \lambda_2 = 0$ we see that $a_3 = \psi(E_{34}) \neq n-1$ (else $a_2=a_4 = a_1 = a_5=0$, contradicting $F = 3n$); moreover  $a_4 = \psi(E_{23}) \neq n-1$ and $a_5 = \psi(E_{13}) \neq n-1$. In fact $a_4 = n-1$ implies  $a_2 = a_3 =0$, contradicting $\la_3  = 2n$, similarly $a_5 = n-1$  contradicts $\la_3  = 2n$. 

Hence 
$\{  E_{15},  E_{25}, E_{34},  E_{23} ,   E_{13}  \}\subset M$ and we are done.

Consider the next case (12), where we can apply again remark \ref{intersection} to infer that property (3) is satisfied; if $\rk\{D_{\sigma} : \psi(\sigma) \neq -1\} \cup \{E_{i5}\}= 5$,  we can apply Theorem \ref{gvt}. If instead $\rk\{D_{\sigma} : \psi(\sigma) \neq -1\}  \cup \{E_{i5}\}< 5$, then wlog we can assume that $i=4$ and it is easy to see that $\psi =(n-1,n-1,n-1,1,1)$. 
Since this case occurs for $n=5$,  we are done.

\subsection{{\color{darkgreen}$F(a_1, \ldots , a_5) = 4n$}}
The first case is:
\begin{itemize}
\item[(13)] $K_Y + \mathcal{L}_{\psi} \equiv L - E_1 -E_2-E_3-E_4 \equiv  E_{34}-E_{35} -E_{45}$.
\end{itemize}

Here the divisors appearing in $A$ and $B$ are disjoint,  conditions (1)-(2) of Theorem \ref{gvt} are fulfilled. Moreover, since $\lambda_i =2n$ for all i, we have $\psi(E_{i5}) \neq n-1$ for all $i \in \{1,2,3,4\}$, hence also condition (3) is fulfilled.

We have, for $i=1,2,3$,  $K_Y + \mathcal{L}_{\psi} \equiv E_{i4}-E_{i5} -E_{45} = E_{34}-E_{35} -E_{45}$. 

Then
$$
\{E_{ij} : E_{ij}.E_{35} = E_{ij}.E_{45} = 0\} = \{E_{15}, E_{25}, E_{34} \}.
$$
Note that this set has rank 3 and is equal to $\{D_{\sigma} : \psi(\sigma) \neq n-1, D_{\sigma}E_{35} = D_{\sigma}E_{45} = 0 \} \cup \{E_{34} \}$. 

Therefore we are done if we can  show that $R=0$ (cf. Theorem \ref{gvt}, (4)) for some choice of $ i \in \{1,2,3\}$.
The only case that does not work is the case where $\psi(E_{ij}) \neq n-1$ for all $i,j  \in \{1,2,3\}$. To handle this case we consider: 
\begin{itemize}
\item $\mathcal{F'} := \Omega_Y^1( \log E_{ij}: (i,j) \neq (1,2) , \ \psi(E_{ij} )\neq n-1)(E_{34}-E_{35} -E_{45})$,
\item $\mathcal{F} := \Omega_Y^1( \log E_{ij}, \psi(E_{ij} ) \neq n-1)(E_{34}-E_{35} -E_{45})$.
\end{itemize}

We have the exact sequence 
$$
0 \rightarrow \mathcal{F'} \rightarrow \mathcal{F} \rightarrow \hol_{E_{12}}(E_{34}-E_{35} -E_{45}) \cong \hol_{\PP^1}(-1) \rightarrow 0.
$$
Therefore it suffices to show that $H^1(Y, \mathcal{F'}) = 0$, which is true since the assumption (4) of Theorem \ref{gvt} is satisfied for $\mathcal{F'}$.

Consider case
\begin{itemize}
\item[(14)] $K_Y + \mathcal{L}_{\psi} \equiv L  -E_i-E_j-E_k$ for $i\neq j \neq k  \in \{1,2,3,4\}$.
\end{itemize}
Wlog we can assume $(i,j,k) = (1,2,3)$, hence 
$$
K_Y + \mathcal{L}_{\psi} \equiv E_{i4} -E_{i5}, \ i \in \{1,2,3\}.
$$

Since $\lambda_i=2n$ for $i=1,2,3$, we have $\psi(E_{i5}) \neq n-1$ for $i=1,2,3$. We easily see that it suffices  to verify  condition (4) of Theorem \ref{gvt}. But since $\lambda_4 = n$ we have that $a_4 + a_5 + a_6 <2n$ whence we can assume wlog $a_4 \neq n-1$ and choose $i=3$ in the decomposition $A-B$. Therefore
$$
E_{25}, E_{15}, E_{23} , E_{34} \in \{D_{\sigma} : \psi(\sigma) \neq n-1, D_{\sigma}E_{35} = 0 \} \cup \{E_{34} \},
$$
and we are done.

In case 
\begin{itemize}
\item[(15)] $K_Y + \mathcal{L}_{\psi} \equiv L -E_i - E_j$, for $i \neq j  \in \{1,2,3,4\}$
\end{itemize}
wlog $K_Y + \mathcal{L}_{\psi} \equiv E_{34} = A$.  Conditions (1), (2) of  Theorem \ref{gvt}
are obviously true ($B=0$), condition (3) is satisfied since $\psi (E_{15}), \psi (E_{15}) \neq n-1$ (as $\la_1, \la_2 = 2n$).

 We are done if $\rk\{D_{\sigma} : \psi(\sigma) \neq -1\} = 5$. 
If instead $\rk\{D_{\sigma} : \psi(\sigma) \neq -1\} < 5$, by the proof of proposition \ref{rank} we have that  $\psi =(n-1,n-1,3,n-1,n-1)$; since this case occurs for $n=5$
we are done.

The last case of this subsection is

\begin{itemize}
\item [(16)] $K_Y + \mathcal{L}_{\psi} \equiv L-E_i \equiv X_i$ for $i  \in \{1,2,3,4\}$.
\end{itemize}

Here  $\rk\{D_{\sigma} : \psi(\sigma) \neq -1\} = 5$ by proposition \ref{rank} since $X_i$ is not irreducible. We now  verify the assumptions of Theorem \ref{gvt}. Wlog 
$i=1$ and $K_Y + \mathcal{L}_{\psi}  \equiv E_{34} +E_{25}$. (1) and (2) are clear and (3) is verified for $E_{34}$ since $\psi(E_{15}) \neq n-1$ and $E_{34}E_{15} = 1$.
If condition (3) is not verified for $E_{25}$, this means that $\psi = n-1$ for all lines through the point $P_2$. We can however vary the decomposition $X_1 \equiv 
E_{24} +E_{35} \equiv E_{23} +E_{45}$, and if for each (3) does not hold for $E_{35}$, respectively $E_{45}$, the value of $\psi$ equals $n-1$ for all the lines,
contradicting $F = 4n$ (as $n \geq 4$).

\subsection{{\color{darkgreen}$F(a_1, \ldots , a_5) = 5n$}}
Here we have
\begin{itemize}
\item [(17)] $K_Y + \mathcal{L}_{\psi} \equiv 2L- E_1-E_2-E_3-E_4 \equiv X_5$.
\end{itemize}
We know that $\rk\{D_{\sigma} : \psi(\sigma) \neq -1\} = 5$, and $K_Y + \mathcal{L}_{\psi} \equiv E_{14}+E_{23}$. Also here we verify the assumptions of Theorem \ref{gvt}: (1) and (2) are clear and (3) is verified since  $\psi(E_{i5}) \neq -1$ for all $i$.

This concludes the proof of the Main Theorem \ref{main}.

\section{Proof of Theorem \ref{main} for $n = 4,6$}

The result follows from the following:
\begin{proposition}\label{346}
Let $S$ be the  Kummer covering of $Y$ branched in $D$ of exponent $n$, where $n \in \{3,4,6\}$ . Then 
\begin{enumerate}
\item $H^1(S, \Theta_S) = 0$ for $n=4,6$.
\item $H^1(S, \Theta_S) \neq 0$ for $n=3$.
\end{enumerate}

\end{proposition}

\begin{proof}
For assertion 2) consider $\psi =(2,2,2,1,1)$. Then $K_Y+\mathcal{L}_{\psi} \equiv 0$. Moreover, 
$$
\{D_{\sigma} : \psi(\sigma) \neq -1\} = \{E_{23},E_{12}, E_{13}, E_{45}\}.
$$
Then 
$$
\pi_*(\Omega_S^1 \otimes \Omega_S^2)^{\psi} = \Omega_Y^1( \log E_{23}, \log E_{13}, \log E_{12}, \log E_{45}) =: \mathcal{F},
$$
and $\chi(\mathcal{F}) = -1$.

1) For $n=6$ we have to show that (cf. Remark \ref{rank2}, 2))
$$
H^1(Y, \Omega_Y^1( \log E_{15}, \log E_{25}, \log E_{35}, \log E_{45}) (E_{14}+E_{23})) =0.
$$

But this follows from Theorem \ref{gvt}.

For $n=4$, we need to show that (cf. Remark \ref{rank2}, 3))
\begin{itemize}
\item $H^1(Y, \Omega_Y^1( \log E_{14}, \log E_{12},\log E_{15}, \log E_{25}, \log E_{45}) (E_{35}-E_{15}))=0$;
\item $H^1(Y, \Omega_Y^1( \log E_{14}, \log E_{12},\log E_{15}, \log E_{25}, \log E_{45})  (E_{34}-E_{45}))=0$;
\item $H^1(Y, \Omega_Y^1( \log E_{23}, \log E_{34}, \log E_{24}) (E_{15})) =0$.
\end{itemize}

The  vanishing of the first two cohomology groups follows from Theorem \ref{gvt}. 

For the second case observe that for $n= 5$ and $\psi = (4,1,1,1,4)$ we have 
$$
\pi_*(\Omega_S^1 \otimes \Omega_S^2)^{\psi} = \Omega_Y^1(  \log E_{23}, \log E_{34}, \log E_{24}, \log E_{15}) (E_{15}) ) =: \mathcal{F}.
$$

We kow that $h^0(\mathcal{F}) = h^1(\mathcal{F})=0$ and since 

$$
 \Omega_Y^1( \log E_{23}, \log E_{34}, \log E_{24}) (E_{15}) \subset \mathcal{F},
$$
it follows that 
$$
h^0(Y, \Omega_Y^1( \log E_{23}, \log E_{34}, \log E_{24}) (E_{15})) = 0 = \chi( \Omega_Y^1( \log E_{23}, \log E_{34}, \log E_{24}) (E_{15})).
$$
This proves the claim.

\end{proof}

\section{Iterated Campedelli-Burniat type configurations}

Recall that the configuration of the complete quadrangle consists of four points which we can describe as  the vertices of
an equilateral triangle, $e_1 : = (1,0,0),  e_2 : = (0,1,0), e_3 : = (0,0,1)$, plus the barycentre $e_4 : = (1,1,1)$;
the six lines joining them pairwise can be described as the sides $L_i : =\{ x_i=0\}, i=1,2,3,$ of the triangle plus   the three medians
$\Lam_i = \{ x_j = x_k \}, \{ i,j,k\} = \{1,2,3\}$.

The configuration of the complete quadrangle yields a configuration of six  lines and seven points, once we add the three middle points of the sides
$$ e'_i  \in L_i , \ x_i = 0, \ x_j = x_k = 1.  $$ 
These in turn form an equilateral triangle with barycentre $e_4$, and we can define the contraction linear map as the unique projectivity satisfying
$$ A : \PP^2 \ra \PP^2, \ A(e_i ) = e'_i , \  A (e_4) = e_4 .$$ 
$A$ acts on the dual space of linear forms by
$$ A (x_i ) = - x_i + x_j + x_k, $$
hence, as one can easily verify,
$$  A (\Lam_i) = \Lam_i, \ \forall i=1,2,3. $$

\begin{definition}
We define inductively 
\begin{itemize}
\item
$e_i(n) : = A^n (e_i)$,
\item
$ L_i (n) : = A^n (L_i)$,
\item
$\sC_{CB} (0) : = ( \bigcup_1^3 L_i) \cup (\bigcup_1^3 \Lam_i)$,
\item
$\sC_{CB} (n) : = ( \bigcup_{ 0 \leq m \leq n, i=1,2,3} A^m  L_i) \cup (\bigcup_1^3 \Lam_i) = A (\sC_{CB} (n-1)) \cup \sC_{CB} (0)$, \item
$a_n := (1 - (-2)^{n+1}) = a_{n-1} + 3 (-2)^n$,
\item
$ 3 b_n : =| 1 - (-2)^n|= | a_{n-1}| = (-1)^{n+1} a_{n-1}$,
\item
$ c_n : = b_n + (-1)^n$.

\end{itemize}
We shall call $\sC_{CB} (n)$ the {\em nth iterated Campedelli-Burniat configuration}.
\end{definition}

We can write explicit formulae as follows, keeping in mind the $\mathfrak S_3$- invariance of the configurations. 

\begin{prop}
$$ L_i (n) = \{ a_n x_i + a_{n-1} (x_j + x_k) = 0\}.$$
$$ e_1(n) = ( c_n, b_n, b_n ) .$$
\end{prop}

\begin{proof}
For $n=0$ the formulae are true, and inductively
$$ A ( a_n x_i + a_{n-1} (x_j + x_k) ) = a_n ( - x_i + x_j + x_k) +  a_{n-1} 2 x_i =$$ 
$$ = (-a_n + 2 a_{n-1} ) x_i +  a_{n} (x_j + x_k) =  (a_n + 3 (-2)^{n+1}) x_i +  a_{n} (x_j + x_k) .$$
Also, $$  A (e_1(n) = A ( c_n, b_n, b_n) = ( 2 b_n, c_n + b_n, c_n + b_n),$$
and it suffices to see that $c_{n+1} = 2 b_n $, $c_n + b_n = b_{n+1} $.

Indeed, $$3 (c_n + b_n)=  2 | 1 - (-2)^n| + 3 (-1)^n = (-1)^{n-1} ( 2 - 2 (-2)^n -3 )=  (-1)^{n-1} (-1 + (-2)^{n+1}) =$$
$$=  (-1)^n  (1 - (-2)^{n+1})  = 3 b_{n+1} ,$$
then clearly $$ 2 b_n = c_n + b_n - (-1)^n =  b_{n+1} + (-1)^{n+1} = c_{n+1} .$$
\end{proof}

To finish the description of the configuration, we determine the intersection points of the  lines of the configuration $\sC_{CB}(n)$.

\begin{itemize}
\item
$\Lam_i \cap \Lam_j = \{ e_4\}, \ i\neq j$
\item
$\Lam_1 \cap L_1(n) = \{ (-2 a_{n-1}, a_n, a_n)\} =  \{ (2 b_n, b_{n+1}, b_{n+1})\} = \{ e_1 (n)\} $
\item
$\Lam_2 \cap L_1(n) = \{ ( a_{n-1},  - (a_n + a_{n-1}), ,a_{n-1}) \}= \\
= \{ ( (-1)^{n+1} 3 b_n , -2 ( 1 - (-2)^{n-1}),  (-1)^{n+1} 3 b_n ) =
( 3 b_n , (-1)^n 2 ( 1 - (-2)^{n-1}),  3 b_n ) 
=(3 b_n ,   2 \cdot 3 \cdot b_{n-1}   ,  3 b_n) = (b_n, 2 b_{n-1} , b_n)$,
\item
$\Lam_1 \cap L_1 = \{ (0,1,-1)\}$,
\item
for $i \neq 1$ $\Lam_1 \cap L_i = \{ ( -a_{n-1}, , 0, a_n)\} = \{ (b_n, 0, , b_{n+1})\}$. 

\end{itemize}

Defining now $\hat{e}_1 : = (0,1,-1)$ and analogously $\hat{e}_i$, we have

\begin{prop}
The configuration $\sC_{CB} (n)$ consists of $ 3 (n+2)$ lines.
The points $e_i, i=1,2,3, 4$ are triple points, while the points $e_i(m), \ m \leq n$ are quadruple points,
the points $e_i(n+1)$ are double points; while through the points $\hat{e}_i$
pass $(n+1)$ lines of the configuration, finally the points $L_i(m) \cap L_j , \ m \leq n, i \neq j$
and their transforms under $A^i$, $ i + m \leq n$ are double points.

Hence, if $t_m$ is the number of vertices of the configuration with valency equal to $m$, we have
\[ t_2 = 3 n (n-1) + 3, t_3 = 4, t_4 = 3n, t_{n+1} = 3.\] 
\end{prop}
\begin{proof}
It suffices to use the previous formulae, observing that the only point which is fixed by $\mathfrak S_3$ is the point $e_4$.

Observe also that $ t_2 + 3 t_3 + 6 t_4 + 3 \frac{1}{2}  n  (n+1) = \frac{1}{2} ( 3 (n+2) ( 3 (n+2) -1)).$

\end{proof}

\begin{center}

\scalebox{1.3} 
{
\begin{pspicture}(0,0)(10,9)
 \psline(0,0)(10,0)(5,8.66)(0,0)
 \psline(0,0)(7.5,4.33)
  \psline(10,0)(2.5,4.33)
 \psline(5,0)(5,8.66)
  \psline(7.5,4.33)(2.5,4.33)(5,0)(7.5,4.33)
   \psline(5,4.33)(3.75,2.165)(6.25,2.165)(5,4.33)
   \psline[linecolor=red](5,2.165)(4.375, 3.2485 )(5.625,3.248)(5,2.165)
   
   \usefont{T1}{ptm}{m}{n}
\rput(-.5,0){$e_1$}
\rput(10.5,0){$e_2$}
\rput(5,9){$e_3$}

\rput(8.2,4.33){$e_1(1)$}
\rput(1.8,4.33){$e_2(1)$}
\rput(5,-.5){$e_3(1)$}

\rput(3,2.165){$e_1(2)$}
\rput(7,2.165){$e_2(2)$}
\rput(4.5,4.6){$e_3(2)$}

\rput(6.3,3.2485){$e_1(3)$}
\rput(3.7,3.2485){$e_2(3)$}
\rput(5.5,1.8){$e_3(3)$}
\end{pspicture} 
}
\end{center}


\end{document}